\documentclass[a4paper]{article}
\usepackage{multirow}
\usepackage{amsfonts}
\usepackage{amsmath}
\usepackage{amssymb}  
\usepackage{graphicx}
\usepackage{subcaption}
\usepackage{hyperref}
\usepackage{listings}
\usepackage{caption}

\usepackage{mathrsfs}
\usepackage{upgreek}
\usepackage{amsthm}
\usepackage{booktabs}

\usepackage[utf8]{inputenc}
\usepackage{bm}
\usepackage{amsmath}
\usepackage{calc}
\usepackage{enumerate}
\usepackage{epstopdf}
\usepackage{placeins}
\usepackage{xcolor}

\usepackage{float}
\usepackage{stmaryrd}
\usepackage{bm}
\usepackage{geometry}
\usepackage{bbold}
\usepackage{setspace}

\newcommand{\half}{\frac{1}{2}}

\newcommand{\reals}{\mathbb{R}}

\newcommand{\nex}{\bm{x}}
\newcommand{\ney}{\bm{y}}

\newcommand{\norm}[1]{\left\Vert #1\right\Vert}

\newcommand{\jump}[1]{\llbracket {#1}  \rrbracket}
\newcommand{\av}[1]{\left\{\!\!\left\{   #1 \right\}\!\!\right\} }
\newcommand{\K}{\mathsf{K}}
\newcommand{\V}{\mathsf{V}}
\newcommand{\W}{\mathsf{W}}

\newcommand{\A}{\mathsf{A}}

\newcommand{\X}{\mathsf{X}}

\renewcommand{\S}{\mathsf{S}}
\newcommand{\D}{\mathsf{D}}
\newcommand{\R}{\mathsf{R}}
\newcommand{\E}{\mathsf{E}}
\newcommand{\IR}{\mathbb{R}}

\renewcommand{\r}{|\bm{x} - \bm{y}|}

\newcommand{\dbtilde}[1]{\widetilde{\raisebox{0pt}[0.85\height]{$\widetilde{#1}$}}}

\newtheorem{theorem}{Theorem}[section]

\newtheorem{remark}[theorem]{Remark}

\newtheorem{problem}[theorem]{Problem}

\begin{document}

\title{Time-Domain Multiple Traces Boundary Integral Formulation for Acoustic Wave Scattering in 2D}
\author{Carlos Jerez-Hanckes\thanks{(carlos.jerez@uai.cl), Faculty of Engineering and Sciences, Universidad Adolfo Ib\'a\~nez, Santiago, Chile}
	\and Ignacio Labarca\thanks{(ignacio.labarca@sam.math.ethz.ch), Seminar for Applied Mathematics, Swiss Federal Institute of Technology, Zurich, Switzerland
}}

\maketitle

\begin{abstract}
We present a novel computational scheme to solve acoustic wave transmission problems over two-dimensional composite scatterers, i.e.~penetrable obstacles possessing junctions or triple points. The continuous problem is cast as a multiple traces time-domain boundary integral formulation. For discretization in time and space, we resort to convolution quadrature schemes coupled to a non-conforming spatial spectral discretization based on second kind Chebyshev polynomials displaying fast convergence. Computational experiments confirm convergence of multistep and multistage convolution quadrature for a variety of complex domains.

\end{abstract}

\section{Introduction}
\label{sec:intro}

We are interested in solving acoustic wave transmission problems arising from the scattering by composite objects in two dimensions. More precisely, we consider a bounded Lipschitz domain $ \Omega\subset\IR^2$, composed of $ M $ non-overlapping Lipschitz subdomains  $ \Omega_i$, $i = 1, \ldots, M,$ such that
\begin{equation}\label{eq:subdomains}
 \overline{\Omega} = \bigcup_{i = 1}^{M} \overline{\Omega}_{i},
 \end{equation}
where $ \Omega_i \cap \Omega_j = \emptyset $ for $ i\neq j $. We define $ \Gamma_{ij}:= \partial \Omega_i \cap \partial \Omega_j $ for the interface between domains $ \Omega_i $ and $ \Omega_j $, with $\Gamma_{ij}=\Gamma_{ji}$. We also denote by $ \Omega_0 := \mathbb{R}^d \ \backslash \  \overline{\Omega} $ the unbounded exterior domain. Notice that, for $i = 0, \ldots, M,$ one can write 
$$ \Gamma_i := \partial \Omega_i = \bigcup_{j \in \Upsilon_i} \Gamma_{ij}, $$
where $ \Upsilon_i $ corresponds to an index set defined as
$$ \Upsilon_i := \left\{  j \in \{0,\ldots, M\} : \quad j \neq i \text{ and } \Gamma_{ij} \neq \emptyset \right\}.$$
The above composite material is characterized by piecewise-constant coefficients, $c_{i} > 0 $ corresponding to the wavespeed on the domain $ \Omega_i$, $i=0,\ldots,M$. Assuming some incoming wave with no volume sources denoted $u^\text{inc}$, we set $u_{i}:= u|_{\Omega_{i}} $ for the total wave inside $\Omega_i$, for $i=1,\ldots,M$, and by $u_0:= u|_{\Omega_{0}} $ the scattered wave in the exterior domain. With this, we seek to solve the time-domain acoustic transmission problem:
\begin{equation}\label{eq:wave-eqn}
\left\{ 
\begin{array}{ll}
\dfrac{1}{c_{i}^2} \dfrac{\partial^2 u_{i}}{\partial t^2} - \Delta u_{i} = 0 & \text{ in }\Omega_{i} \times [0, \infty), \quad i = 0, \ldots, M,\\
&\\
 u_0 - u_j = -u^{\text{inc}} & \text{ on }\Gamma_{0j} \times [0, \infty), \ j\in \Upsilon_0,\\
&\\
 \dfrac{\partial u_0}{\partial n^0} +   \dfrac{\partial u_j}{\partial n^j} = -\dfrac{\partial u^{\text{inc}}}{\partial n^0} & \text{ on }\Gamma_{0j} \times [0, \infty),  \ j\in \Upsilon_0,\\
&\\
 u_i - u_j = 0 &\text{ on } \Gamma_{ij} \times (0, \infty],\ i,j>0, \ j \in \Upsilon_i,\\
 &\\
 \dfrac{\partial u_i}{\partial n^i} +  \dfrac{\partial u_j}{\partial n^j} = 0& \text{ on } \Gamma_{ij} \times [0, \infty), \ i,j>0,\ j \in \Upsilon_i,\\
 &\\
 u_i = \dfrac{\partial u_{i}}{\partial t} = 0 & \text{ in }\Omega_i \times \{0\}, i = 0, \ldots, M.
\end{array}
\right.
\end{equation}
\begin{figure}[t]
	\centering 
	\includegraphics[width=0.6\linewidth]{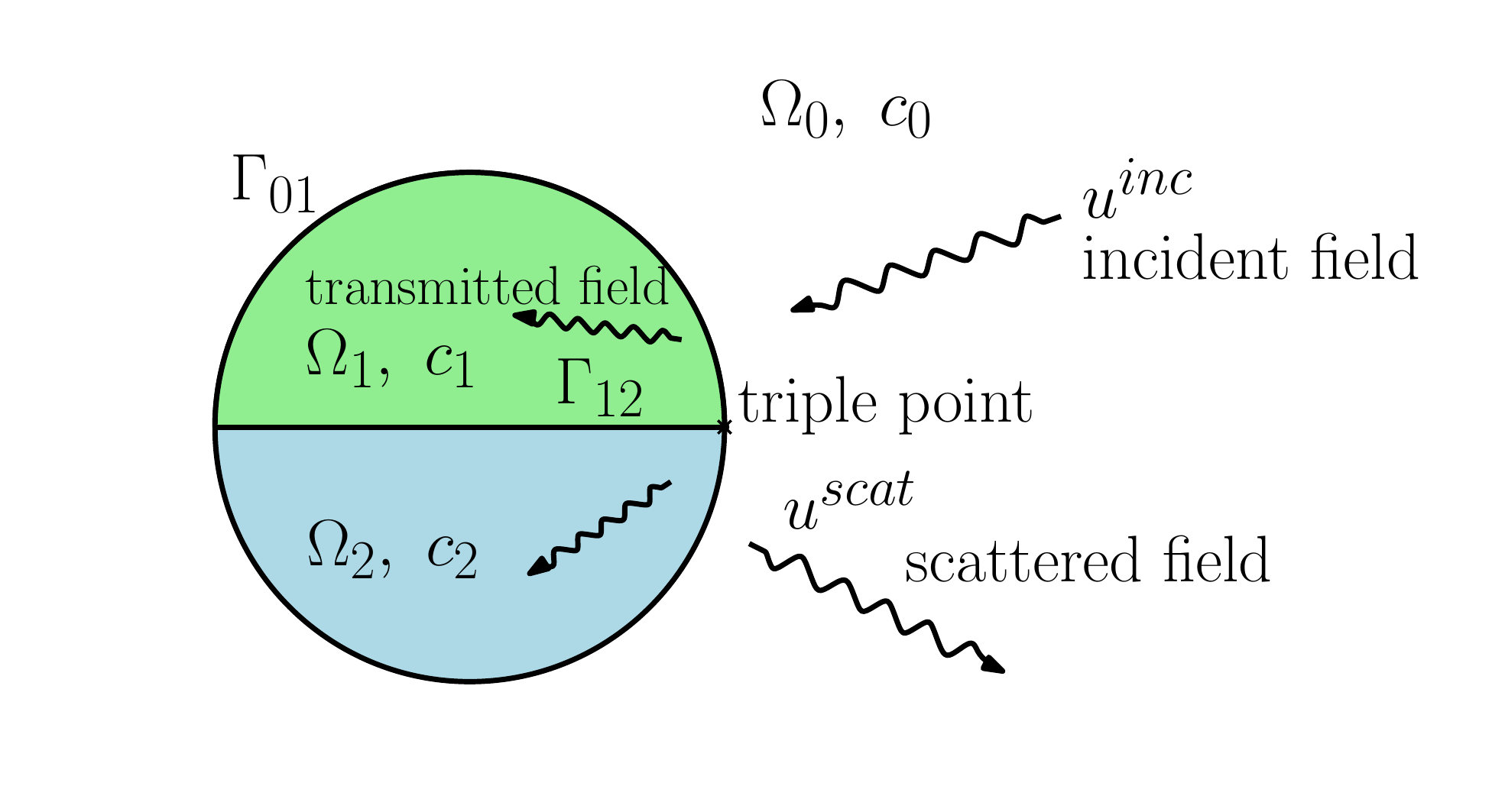}
	\label{geom}
	\caption{Example of a scatterer with two homogenous subdomains.}
\end{figure}
We solve the above wave scattering problem by combining the following approaches:
\begin{enumerate}[(i)]
	\item Boundary integral equations (BIEs) in the form of the local multiple traces formulation (MTF) in space-time domain;
	\item Spectral non-conforming Galerkin discretizations with Chebyshev polynomials for spatial discretization of the BIEs; and,
	\item Convolution quadrature (CQ) for approximation in time.
\end{enumerate}

BIEs lead to unknowns defined on boundaries while rigorously enforcing causality or radiation conditions in the time-harmonic case \cite{mclean2000strongly}. Specifically, for composite materials one can resort to BIEs based on single-trace formulations or different versions of MTFs according to the strong or weak enforcement of transmission conditions (\emph{cf.}~\cite{hiptmair2012multiple,CHJ13,CHJ15}). The discretization of MTFs then may be carried out via boundary elements (BE) \cite{hiptmair2012multiple,HJA16}, spectral Galerkin \cite{jerez2015local,jerez2016local,HJH18} or Nystr\"om \cite{jerez2017multitrace} methods. 

In the following, we employ a spatial spectral non-conforming discrete Galerkin scheme for the (local) MTF based on Chebyshev polynomials in order to attain accurate approximations with a small number of degrees of freedom. Moreover, this setting allows for the efficient computation of matrix entries via the Fast Fourier Transform (FFT) by relating Chebyshev coefficients to Fourier coefficients as well as direct implementation of compression techniques (cf.~\cite[Section~3.3]{jerez2015local}, \cite{jerez2016local} and \cite[Chapter~3]{trefethen2013approximation}). First and second kind single-trace formulations \cite{spindler2015second,EFH19} also exist but there are no available spectrally convergent methods for them. Still, due to the non-conforming nature of our spectral scheme, its numerical analysis remains open an open problem as existing results require Galerkin discretization in standard Sobolev spaces. Thus, the present work focuses on algorithmic aspects. Indeed, though we introduce the required functional spaces for rigorously formulating the MTF, these are not essential for neither the current presentation nor its implementation. For a precise derivation and results of existence and uniqueness of the continuous problem we refer to \cite{hiptmair2012multiple}.

For the time-domain approximation we opt for CQ methods \cite{banjai2012wave,sayas2016retarded} due to their stability and amenability to be coupled to any available complex frequency-domain solvers such as the ones discussed above. Algorithmically, we refer to \cite{hassell2016convolution} for the corresponding pseudo-code to efficiently compute forward convolutions and solve the arising equations. A thorough analysis of multistage CQ can be found in \cite{banjai2011error, banjai2011runge, banjai2012runge, lubichRK}, while generalized CQ allowing for different timesteps is described in \cite{lopez2013generalized, lopez2016generalized}. CQ implementation and analysis for wave scattering problems have been studied in \cite{banjai2008rapid}, with emphasis on transmission problems provided in \cite{qiu2016time,qiu2016costabel} and composite materials in \cite{rieder2017convolution,EFH19}. Yet, to our knowledge, no time-domain CQ-MTF has been described in the literature, with the current work being the first contribution to that end.

Alternatively, direct Galerkin discretizations for space-time BIEs could also be used to tackle Problem \eqref{eq:wave-eqn}, but this involves the difficult computation of boundary integral operators in time domain \cite{banjai2012wave}. Moreover, as it happens in most time-stepping procedures, a poor choice of time-step may immediately lead to instabilities. Also, long-time computations are often unstable and despite the availability of several remedies (cf.~\cite{davies1998stability,davies1997averaging,davies2004stability,davies2005stability}), CQ remains a simple, efficient and stable method to discretize time-domain problems without major complications. 

The present manuscript is structured as follows. In Section \ref{sec:MTF}, we present the MTF in frequency domain as in \cite{hiptmair2012multiple}. Section \ref{sec:TDBIE} recalls the main ideas behind time-domain BIEs. Section \ref{sec:numscheme} describes the spatial spectral non-conforming Galerkin discretization for MTF \cite{jerez2015local, jerez2016local}. Then, in Sections \ref{sec:multistep} and \ref{sec:multistage} we explain CQ for multistep and multistage linear methods and their application to convolutional equations in Section \ref{sec:conv_eq}. Finally, in Section \ref{sec:numexp} we show some numerical experiments for different geometries and parameters of interest, which clearly illustrate the capabilities and limitations of the proposed methodology and sketch future research ideas. 

Besides the definitions already introduced, the notation used throughout can be summarized as follows:
\begin{itemize}
\item The imaginary unit is denoted $\imath =\sqrt{-1}$.
	\item $ u $ denotes an acoustic field in space and time (see \eqref{eq:wave-eqn}).
	
	\item $ U $ and $ V $ denote Laplace-domain fields related to the modified Helmholtz equation (see Section \ref{sec:MTF}).
	
	\item $ \lambda, \varphi, \psi $ denote time-domain boundary densities, with values in Sobolev spaces defined over a boundary or interface (see Section \ref{sec:TDBIE}, \eqref{eq:tdpotentials} and Problem \ref{TDMTF}).
	
	\item $ \Lambda, \Phi, \Psi $ are boundary functions in frequency domain (see Section \ref{sec:MTF}, \eqref{eq:duality_prod}, \eqref{eq:potentials}).
	
	\item Bold variables $ \bm{\lambda}, \ \bm{\varphi}, \ \bm{g}, \ \bm{\Lambda}, \ \bm{\Phi}, \ \bm{G} $ denote elements of Cartesian product spaces (see Problems \ref{MTF_freq}, \ref{TDMTF}, \ref{specMTF}, \ref{CQMTF} and \ref{discreteTDMTF}).  
	
	\item Operators $ \A, \S, \D,  \K, \K', \V, \W, \X, \E, \R,  \mathsf{L} $ are defined in the frequency domain, sometimes depending on a complex parameter $ s\in\mathbb{C}_+ $ (see Section \ref{sec:MTF}, \eqref{eq:potentials}, \eqref{eq:bios}, \eqref{eq:cald_operator}).
	
	\item Operators $ \mathcal{A}, \mathcal{S}, \mathcal{D}, \mathcal{K}, \mathcal{K}', \mathcal{V}, \mathcal{W}, \mathcal{F} $ are defined in the time domain by means of the inverse Laplace transform of frequency domain counterparts (see Section \ref{sec:TDBIE}, \eqref{eq:tdpotentials}, \eqref{eq:td_op}, \eqref{eq:cald_op_td} and Problem \ref{TDMTF}).
		
	\item $ N $ is the number of timesteps (see Problem \ref{CQMTF}).
	
	\item $ L $ refers to the number of degrees of freedom of the spectral spatial discretization (see Section \ref{sec:spectral}, Problem \eqref{specMTF}).
	
	\item Variable $ s \in \mathbb{C}_+$ is reserved for the parameter in Laplace domain and modified Helmholtz equations (see Section \ref{sec:MTF}, \eqref{eq:modHem}, \eqref{eq:potentials}, \eqref{eq:representation_formula}, \eqref{eq:bios}).
	
\end{itemize}

\section{Multiple traces formulation}
\label{sec:MTF}
We recall results for the (local) MTF for Helmholtz scattering problems over composite materials as exposed in \cite{hiptmair2012multiple, jerez2016local}. Despite being initially developed for time-harmonic problems, the formulation lends itself easily to account for time-dependent wave scattering via the inverse Laplace transform as discussed in Section \ref{sec:TDBIE}.

Let $\Omega\subset \mathbb{R}^2 $ be as in \eqref{eq:subdomains} and $ \bm{n}^{i}$ denote the unit outward normal vector to $ \partial\Omega_i$. Trace spaces on the boundary are written $ H^{\pm 1/2}(\partial\Omega_{i}) $. For an open interface $ \Gamma_{ij}$, we introduce
$$ H^{\frac{1}{2}}(\Gamma_{ij}) := \{\xi|_{\Gamma_{ij}} \ : \ \xi\in H^{\frac{1}{2}}(\partial\Omega_i)   \} \quad  \text{ and } \quad \widetilde{H}^{\frac{1}{2}}(\Gamma_{ij}) := \{ \eta \in H^{\frac{1}{2}}(\Gamma_{ij}) \ : \ \widetilde{\eta}\in H^{\frac{1}{2}}(\partial\Omega_i)    \}, $$
where $ \tilde{\eta} $ is the extension by zero of $ \eta $ over $ \partial\Omega_i $. We identify $ \widetilde{H}^{-\frac{1}{2}}(\Gamma_{ij}) $ with the dual space of $   H^{\frac{1}{2}}(\Gamma_{ij})$. Cartesian product trace spaces over closed boundaries $ \partial \Omega_{i} $ are denoted
$$ \mathbf{V}_{i} := H^{\half}(\partial \Omega_{i}) \times H^{-\half}(\partial \Omega_{i}). $$
For the multiple interfaces case we require the following piecewise or broken spaces:
\begin{equation}\label{eq:tildepw}
\widetilde{H}_\text{pw}^{\mp \half}(\partial \Omega_i):= \left\{ V\in H^{\mp \half}(\partial \Omega_i)\ : \ V|_{\Gamma_{ij}} \in \widetilde{H}^{\mp \half}(\Gamma_{ij}), \quad \forall j\in \Upsilon_i   \right\},
\end{equation}
whose respective duals are
\begin{equation}\label{eq:pw}
H_\text{pw}^{\pm \half}(\partial \Omega_i):= \left\{ V\in \mathcal{D}^\prime(\partial \Omega_i)\ : \ V|_{\Gamma_{ij}} \in H^{\pm \half}(\Gamma_{ij}), \quad \forall j\in \Upsilon_i   \right\},
\end{equation}
where $\mathcal{D}^\prime(\partial \Omega_i)$ denotes the space of distributions over $\partial \Omega_i$. With \eqref{eq:tildepw} and \eqref{eq:pw} we define the broken Cartesian product spaces:
\begin{equation*}
\mathbf{V}_{\text{pw}, i}:= H^{\half}_\text{pw}(\partial \Omega_i) \times H^{-\half}_\text{pw}(\partial \Omega_i) \quad \text{ and }\quad 
\dbtilde{\mathbf{V}}_{i}:= \widetilde{H}^{\half}_\text{pw}(\partial \Omega_i) \times \widetilde{H}^{-\half}_\text{pw}(\partial \Omega_i).
\end{equation*}
Furthermore, we build the next spaces defined over $M\in\mathbb{N}$ subdomain boundaries: 
\begin{equation*}
\begin{array}{rcl} 
\dbtilde{\mathbb{V}}_{M} &:=& \dbtilde{\mathbf{V}}_0 \times \ldots \times \dbtilde{\mathbf{V}}_{M},\\
\mathbb{V}_{M} &:=& \mathbf{V}_0 \times \ldots \times \mathbf{V}_{M},\\
\mathbb{V}_{\text{pw},M} &:=& \mathbf{V}_{\text{pw},0} \times \ldots \times \mathbf{V}_{\text{pw},M}.
\end{array}
\end{equation*}
We denote by $ \gamma_d ^i $ and $ \gamma_n^i $ the standard Dirichlet and Neumann trace operators taken from the interior of $ \Omega_i $ and set the vector trace operator $ \gamma^i V := 
(\gamma_d^i V, \gamma_n^i V)$. Exterior counterparts are denoted $ \gamma_d^{i,c}, \gamma_n^{i,c}$ and $ \gamma^{i,c}$. Neumann traces $ \gamma_n^{i,c} $ involve the unit normal pointing towards the interior of $ \Omega^{i,c} $, allowing us to introduce average and jump trace operators over $ \partial \Omega_i $:
\begin{equation}
\jump{\gamma V}_{i} := \begin{pmatrix}
\jump{\gamma_d V}_{i} \\ \jump{\gamma_n V}_{i }
\end{pmatrix} = \begin{pmatrix}
\gamma_d^{i,c}V - \gamma_d^iV \\
\gamma_n^{i,c}V - \gamma_n^iV 
\end{pmatrix}
\ \ \text{and}\ \
\av{\gamma V}_{i} := \begin{pmatrix}
\av{\gamma_d V}_{i} \\ \av{\gamma_n V}_{i }
\end{pmatrix} = \dfrac{1}{2} \begin{pmatrix}
\gamma_d^{i,c}V + \gamma_d^iV \\
\gamma_n^{i,c}V + \gamma_n^iV 
\end{pmatrix},
\end{equation}
Let $ G_{i}(\r; s) = \frac{i}{4}H^{(1)}_0(\imath s_{i}\r) $ denote the fundamental solution of the homogenous modified Helmholtz equation 
\begin{equation}\label{eq:modHem}
 -\Delta U + s_{i}^2 U = 0 \quad \text{in}\  \mathbb{R}^2 \backslash \partial\Omega_{i} , 
\end{equation} 
 where $ s_{i} = sc^{-1}_{i}, \ s\in \mathbb{C}_+ := \{ s\in \mathbb{C} \ : \  \text{Re}(s) > 0\} $, and $ H^{(1)}_0(\cdot)$ denotes the first kind Hankel function of order zero. We define layer potentials acting on sufficiently smooth densities $ \Psi $ and $ \Phi $ on a boundary $ \Gamma_{i} =\partial \Omega_{i} $ as
\begin{equation}
\label{eq:potentials}
\begin{split}
(\S_{i}(s)\Psi)(\bm{x}) \ &:= \displaystyle \int_{\Gamma_{i}}G_{i}(\|\bm{x}-\bm{y}\|; s)\Psi(\bm{y}) \text{d}s_{y},\\ 
\quad (\D_{i}(s)\Phi)(\bm{x}) \ &:= \displaystyle \int_{\Gamma_{i}} \bm{n}^{i}(\bm{y})  \cdot \nabla_{\bm{y}}G_{i}(\|\bm{x}-\bm{y}\|; s)    \Phi(\bm{y}) \text{d}s_{y}.
\end{split}
\end{equation}
The integral representation formula for a solution $ U $ of \eqref{eq:modHem} in $ \mathbb{R}^2 \ \backslash \ \partial \Omega_{i} $ yields
\begin{equation}
\label{eq:representation_formula}
U(\bm{x}) =  (\D_{i}(s)\jump{\gamma_d U}_{i})(\bm{x}) - (\S_{i}(s)\jump{\gamma_n U}_{i})(\bm{x}), \quad \bm{x} \in \mathbb{R}^2 \ \backslash \ \partial \Omega_{i}.
\end{equation}
For each subdomain $ \Omega_{i},$ we introduce corresponding boundary integral operators (BIOs): 
	\begin{equation}\label{eq:bios}
	\begin{array}{lllclcl} 
	\V_{i}(s)&:=& \av{\gamma_d \S_{i}(s)}_{{i}} & : &  H^{-\half}(\partial \Omega_{i}) &\rightarrow& H^{\half}(\partial \Omega_{i}),\\
	\K'_{i}(s)&:=& \av{\gamma_n \S_{i}(s)}_{{i}} & : & H^{-\half}(\partial \Omega_{i}) &\rightarrow& H^{-\half}(\partial \Omega_{i}),\\
	\K_{i}(s)&:=& \av{\gamma_d \D_{i}(s)}_{{i}} & : & H^{\half}(\partial \Omega_{i}) &\rightarrow& H^{\half}(\partial \Omega_{i}),\\
	\W_{i}(s)&:=-&\av{\gamma_n \D_{i}(s)}_{{i}} & : & H^{\half}(\partial \Omega_{i}) &\rightarrow& H^{-\half}(\partial \Omega_{i}).
	\end{array}
	\end{equation}
and the Calder\'on block operator:
\begin{equation}\label{eq:cald_operator}
\A_{i}(s) := \begin{pmatrix}
-\K_{i}(s) & \V_{i}(s) \\ \ \ \W_{i}(s) & \K'_{i}(s)
\end{pmatrix} \ : \ \mathbf{V}_{i} \rightarrow \mathbf{V}_{i}. 
\end{equation}
\begin{remark}
A coercivity result for a scaled version of $ \A_{i}(s) $ can be found in \cite[Proposition~3.2]{EFH19}.
\end{remark}

For $ \bm{\Lambda} = (\Lambda_0, \ldots, \Lambda_M)\in \mathbb{V}_M, \ \bm{\Phi}= (\Phi_0, \ldots, \Phi_M)\in \dbtilde{\mathbb{V}}_M $ we denote $ (\bm{\Lambda}, \bm{\Phi})_{\times} := \displaystyle \sum_{i  = 0}^M (\Lambda_{i}, \Phi_{i})_{\times, i}, $ with the sesquilinear duality product:
\begin{equation}\label{eq:duality_prod}
\left(   {\Lambda}_{i}, {\Phi}_{i}\right)_{\times, i} = \langle \Lambda_{i, d}, \overline{\Phi}_{i, n} \rangle_{\Gamma_i} + \langle \overline{\Phi}_{i,d}, \Lambda_{i,n} \rangle_{\Gamma_i}
\end{equation} for $ {\Lambda}_{i}, {\Phi}_{i} \in \mathbf{V}_{i}$, $ {\Lambda_{i}} = (\Lambda_{i,d}, \Lambda_{i,n}), $ and $ {\Phi_{i}}= (\Phi_{i,d}, \Phi_{i,n}) $.\\
Now, we are able to present the complex frequency-domain MTF.
\begin{problem}[Multiple traces formulation]\label{MTF_freq}
	Seek $  \bm{\Lambda}\in \mathbb{V}_{M}$ such that the variational form
	\begin{equation}
	\label{eq:MTF}
	 (\textbf{\textsf{F}}(s) \bm{\Lambda}, \bm{\Phi})_{\times} = \left( \bm{G}, \bm{\Phi} \right)_{\times}, \quad \text{ for all }\bm{\Phi}\in  \dbtilde{\mathbb{V}}_{M},
	\end{equation}
	is satisfied for $ \bm{G} = (
	{G}_0, {G}_1, \ldots, {G}_{M} ) \in \mathbb{V}_{\text{pw},M}$ with 
	\begin{equation}
	\textbf{\textsf{F}}(s) := \begin{pmatrix}
	\mathsf{A}_0(s) & -\frac{1}{2}\mathsf{\widetilde{X}}_{01} & \ldots & -\frac{1}{2}\mathsf{\widetilde{X}}_{0N}\\
	-\frac{1}{2}\mathsf{\widetilde{X}}_{10} & \mathsf{A}_1(s) & \ldots & -\frac{1}{2}\mathsf{\widetilde{X}}_{1N}\\
	\vdots & \vdots & \ddots & \vdots \\
	-\frac{1}{2}\mathsf{\widetilde{X}}_{N0} & -\frac{1}{2}\mathsf{\widetilde{X}}_{N1} & \ldots & \mathsf{A}_{M}(s)
	\end{pmatrix} \ : \ \mathbb{V}_{M} \rightarrow \mathbb{V}_{M,\text{pw}}.
	\end{equation}
	wherein $\mathsf{\widetilde{X}}_{i k}$ denotes the restriction-and-extension by zero operator:
\begin{equation}\label{eq:tildeX} 
\left(\mathsf{\widetilde{X}}_{i k}\Lambda_{k}, \Phi_{i}\right)_{\times, i} := \left\{ 
\begin{array}{cl}
\left\langle\Lambda_{k, d}, \overline{\Phi}_{i, n}\right\rangle_{\Gamma_{i k}} - \left\langle\overline{\Phi}_{i, d}, \Lambda_{k, n}\right\rangle_{\Gamma_{i k}}, & \quad k\in \Upsilon_{i},\\
&\\
0, & \quad k\notin \Upsilon_{i},
\end{array}
  \right.
\end{equation}
i.e. the duality product on the interface $ \Gamma_{i k} $.
\end{problem}

\begin{remark}
Broken function spaces $\widetilde{H}_\text{pw}^{\mp \half}(\partial \Omega_i)$ and ${H}_\text{pw}^{\pm \half}(\partial \Omega_i)$ --\eqref{eq:tildepw} and \eqref{eq:pw}-- let us properly define the restriction of boundary data to the open interface $ \Gamma_{i k} $ in \eqref{eq:tildeX}. This is also crucial for the analysis of the MTF \cite{hiptmair2012multiple}.
\end{remark}

\begin{remark}\label{rhs}
For the case of an incident complex-valued field $ U^{\text{inc}} $ coming from the exterior domain $ \Omega_0, $ the right-hand side in Problem \ref{MTF_freq} is given by $ \bm{G} = \frac{1}{2}\mathsf{R} \gamma^0 U^{\text{inc}} $, where we summarize in
$$ \mathsf{R} \ : \  \mathbf{V}_0 \rightarrow \mathbb{V}_{pw, M}$$ 
the restriction-extension operators defined in \cite{hiptmair2012multiple} by the expression
$$ 
\begin{array}{rcl}
\left( \mathsf{R} \gamma_0 U^{\text{inc}}, \bm{\Phi} \right)_{\times} & = &  \left\langle \gamma_d^0 U^{\text{inc}}, \overline{\Phi}_{0, n}   \right\rangle_{\Gamma_{0}} - \langle \Phi_{0, d}, \gamma_n^0 \overline{U^{\text{inc}}}  \rangle_{\Gamma_{0}}  \\
&&\\
&- & \displaystyle\sum_{k\in \Upsilon_0} \left\{ \left\langle \gamma_d^0 U^{\text{inc}}, \overline{\Phi}_{k, n}   \right\rangle_{\Gamma_{0k}} + \langle \Phi_{k, d}, \gamma_n^0 \overline{U^{\text{inc}}}  \rangle_{\Gamma_{0k}} \right\}.
\end{array}
$$
\end{remark}

Well-posedness of Problem \ref{MTF_freq} follows from the corresponding result for its real wavenumber counterpart (Helmholtz transmission problem) studied in \cite{hiptmair2012multiple}, based on coercivity results that still hold for the complex wavenumber case.
\begin{theorem}[Existence and Uniqueness]\label{MTF_freq_exist_uniq}
	There exists a unique solution to Problem \ref{MTF_freq}.
\end{theorem}

\section{Time Domain BIEs}
\label{sec:TDBIE}
We now shift our attention to deriving a time-domain formulation based on the inverse Laplace transform of operator $ \textbf{\textsf{F}}(s)$ following \cite{sayas2016retarded}. To this end, let us start by recalling the Laplace transform and its inverse for a causal function $ f \in L^2(\mathbb{R}_+)$:
\begin{equation}
F(s) := (\mathcal{L}f)(s) := \displaystyle \int_{0}^{\infty} f(t)e^{-st} \text{d}t, \quad f(t) = (\mathcal{L}^{-1}\{s\mapsto   F(s)\})(t) := \dfrac{1}{2\pi \imath} \int\limits_{\sigma -i\infty}^{\sigma+i\infty} e^{st}F(s)\text{d}s,
\end{equation}
where $ \sigma >0  $. These definitions can be easily extended to vector-valued causal distributions (cf.~\cite[Section~1.2]{hassell2016convolution} or \cite[Section~2.1]{sayas2016retarded}).
For causal density functions $\psi  :  \reals_+ \rightarrow H^{-\half}(\Gamma)$ and $ \varphi  :  \reals_+ \rightarrow H^{\half}(\Gamma)$ that can be extended to causal tempered distributions with values in Sobolev spaces and their corresponding Laplace transforms $ \Psi  :  \mathbb{C}_+ \rightarrow  H^{-\half}(\Gamma)$ and $ \Phi :  \mathbb{C}_+ \rightarrow H^{\half}(\Gamma)$, we define time-domain single  and double layer potentials as
\begin{equation}\label{eq:tdpotentials}
\mathcal{S}\ast \psi := \mathcal{L}^{-1}\left(\{s\mapsto \mathsf{S}(s)\Psi(s)  \}\right), \quad \mathcal{D}\ast \varphi := \mathcal{L}^{-1}\left(\{s\mapsto \mathsf{D}(s)\Phi(s)  \}\right),
\end{equation}
where $ \mathsf{S}(s) $ and $ \mathsf{D}(s) $ are the single and double layer potentials for the (modified) Helmholtz equation $ -\Delta U + s^2 U = 0$ defined in \eqref{eq:potentials}.

In general, for normed spaces $ X $ and $ Y $ and a transfer function $ \mathsf{F}(s)  : X \rightarrow Y $ in the Laplace domain, we will define the corresponding convolutional operator in time domain as 
\begin{equation}\label{eq:td_op}
\mathcal{F}\ast \varphi := \mathcal{L}^{-1}\left(\{s\mapsto  \mathsf{F}(s)\Phi(s)   \}\right), 
\end{equation}
with $(\mathcal{L}\varphi)(s) = \Phi(s) $ and $ \Phi(s) \in X $ for all $ s \in \mathbb{C}_+ $. Details about the existence of this operator can be found in \cite[Propositions 3.1.1 and 3.1.2]{sayas2016retarded}.
With this, we are able to define Calder\'on BIOs in the time domain by means of the inverse Laplace transform and corresponding operators for the modified Helmholtz equation. 

We define the time-domain Calder\'on operator:
\begin{equation}\label{eq:cald_op_td}
\quad \mathcal{A} := \begin{pmatrix}
-\mathcal{K} & \mathcal{V}\\ \mathcal{W} & \mathcal{K}'
\end{pmatrix}
\end{equation}
with $ \mathcal{V},\ \mathcal{K},\ \mathcal{K}' $ and $ \mathcal{W} $ being the corresponding time-domain counterparts of the weakly singular, double-, adjoint-double layer and hypersingular BIOs based on \eqref{eq:bios} and \eqref{eq:td_op}.

Finally, we can write a time-domain version of the local MTF.
\begin{problem}\label{TDMTF}
	Let $u^{\text{inc}}$ be as in \eqref{eq:wave-eqn}. We seek a causal $ \mathbb{V}_M$-valued distribution $ \bm{\lambda}$, i.e.~$\bm{\lambda}\in C^1([0,\infty);\mathbb{V}_{M})$, such that
	\begin{equation}
	\mathcal{F}\ast \bm{\lambda} = \bm{g},
	\end{equation} 
	where 
	\begin{equation}
\mathcal{F} := 
	\begin{pmatrix}
	\mathcal{A}_0 & -\frac{1}{2}\mathsf{\widetilde{X}}_{01}\otimes \delta_0 & \ldots & -\frac{1}{2}\mathsf{\widetilde{X}}_{0N}\otimes \delta_0\\
	-\frac{1}{2}\mathsf{\widetilde{X}}_{10}\otimes \delta_0 & \mathcal{A}_1 & \ldots & -\frac{1}{2}\mathsf{\widetilde{X}}_{1N}\otimes \delta_0\\
	\vdots & \vdots & \ddots & \vdots \\
	-\frac{1}{2}\mathsf{\widetilde{X}}_{N0}\otimes \delta_0 & -\frac{1}{2}\mathsf{\widetilde{X}}_{N1}\otimes \delta_0 & \ldots & \mathcal{A}_{M}
	\end{pmatrix} \ \text{and} \ \bm{g} = \dfrac{1}{2} \mathsf{R}\gamma^0u^{\text{inc}}.
		\end{equation}
with $\bm{g}$ in the Bochner space $C^1([0,\infty);\mathbb{V}_{\text{pw},M})$, $\delta_0$ denoting the Dirac delta at time $t=0$ and $ \mathsf{R} $ the operator defined in Remark \ref{rhs}.
\end{problem}

\FloatBarrier
\section{Spatial Spectral Discretizations}
\label{sec:numscheme}
We now turn to the approximation of the above system by means of CQ combined with spatial spectral non-conforming elements, starting with the latter. The reason for this is that the CQ requires multiple solves of the Laplace domain MTF system, thereby rendering low-order methods computationally too intensive and even impractical for complex geometries.

\subsection{Spectral Elements}
\label{sec:spectral}
Following \cite{jerez2015local,jerez2016local}, we consider spectral elements based on Chebyshev polynomials to discretize the spatial unknowns in Problem \ref{TDMTF}.

Assume that for each interface $ \Gamma_{ji} $ there exists a $ C^1$-parametrization $\boldsymbol{h}_{ji} $ that maps the nominal segment $ \widehat{\Gamma} := [-1, 1] $ onto $ \Gamma_{ji}$. We derive a non-conforming Petrov-Galerkin discretization of problem \eqref{eq:MTF} by using the space spanned by second kind Chebyshev polynomials of degree $ \ell \in \mathbb{N}_0$:
\begin{equation}
\upsilon_\ell(t) := \dfrac{\sin (\ell+1)\theta}{\sin \theta}, \quad t = \cos \theta, \quad \theta\in[0,\pi],
\end{equation}
as trial functions and test functions spanned by weighted second kind Chebyshev polynomials as follows. 

Let $L\in\mathbb{N}$ indicate the dimension of the finite-dimensional space. On each interface we use the same trial space for both $ H^\frac{1}{2}(\Gamma_{ji}) $ and $ H^{-\frac{1}{2}}(\Gamma_{ji}) $ defined by 
\begin{equation}\label{eq:trial_space}
 \mathbb{T}_L^{ji} := \text{span }\{\xi^{ji}_0, \ldots, \xi^{ji}_{L}   \}, \quad \xi^{ji}_{\ell}(t) := \dfrac{1}{\|\boldsymbol{h}'_{ji}(t)\|}\upsilon_{\ell}(t), \quad t \in \widehat{\Gamma}.
\end{equation}
Similarly, on each interface we use a test space for both $ \widetilde{H}^{\frac{1}{2}}(\Gamma_{ji}) $ and $ \widetilde{H}^{-\frac{1}{2}}(\Gamma_{ji}) $ defined by
\begin{equation}\label{eq:test_space}
\mathbb{Q}_L^{ji} := \text{span }\{\chi^{ji}_0, \ldots, \chi^{ji}_{L}   \}, \quad \chi^{ji}_{\ell}(t) := \dfrac{1}{\|\boldsymbol{h}'_{ji}(t)\|}\upsilon_{\ell}(t)\sqrt{1-t^2}, \quad t \in \widehat{\Gamma}.
\end{equation}
For $k,\ell\in\mathbb{N}_0$, we recall the orthogonality property:
\begin{equation}\label{eq:orthogonal}
\displaystyle \int_{-1}^{1} \upsilon_\ell(t)\upsilon_k(t)\sqrt{1-t^2}dt = 
\left\{ 
\begin{array}{ll}
0 & \ell\neq k,\\
\pi/2 & \ell = k,
\end{array}
\right.
\end{equation}
which jointly with related Fourier-Chebyshev expansions \cite[Chapter~3]{trefethen2013approximation} allow for the fast computation of matrix entries by means of the FFT. 

\begin{remark}
The above property motivates the weight $ \sqrt{1-t^2} $ in our test functions. More importantly, this weight forces test functions to vanish at the endpoints of every interface as required for spaces $ \widetilde{H}^{\frac{1}{2}}(\Gamma_{ji}). $
\end{remark}

Analogously to the continuous case, we define finite-dimensional subspaces of our broken spaces:
\begin{equation}
\begin{array}{lcl}
\mathbf{X}_L^{i} & := & \left\{ \xi \in \mathbf{V}_{\text{pw},i} \ : \ \xi |_{\Gamma_{ji}} \in \mathbb{T}_L^{ji}\times\mathbb{T}_L^{ji}, \quad j \in \Upsilon_{i}    \right\}, \\
&&\\

\mathbf{Y}_L^{i} & := & \left\{ \chi \in \dbtilde{\mathbf{V}}_{\text{pw},i} \ : \ \chi |_{\Gamma_{ji}} \in \mathbb{Q}_L^{ji}\times\mathbb{Q}_L^{ji}, \quad j \in \Upsilon_{i}     \right\}.
\end{array}
\end{equation}
Density of the above spaces in the product ones $ \mathbf{V}_{\text{pw}, i} $ and $ \dbtilde{\mathbf{V}}_{i} $ was established in \cite[Propositions~1-2]{jerez2015local}.

Finally, we introduce
$$\mathbb{X}_{M,L} := \mathbf{X}_L^0 \times \ldots \times \mathbf{X}_L^{M}, \quad  \mathbb{Y}_{M,L}:= \mathbf{Y}_L^{0}\times \ldots \times \mathbf{Y}_L^M.$$

\begin{remark}
Note that our discretization scheme is non-conforming as it allows for discontinuities in the Dirichlet unknowns at the endpoints of every interface, i.e. $\mathbf{X}_L^{i}\not\subset H^\frac{1}
{2}(\partial\Omega_i)\times H^{-\frac{1}{2}}(\partial\Omega_i)$. Yet, as test functions vanish at the endpoints of each interface, the formulation remains well defined (cf.~\cite{hiptmair2012multiple, jerez2015local}).
\end{remark}

We can now define the discrete version of Problem \ref{MTF_freq}:

\begin{problem}[Spectral non-conforming Petrov-Galerkin MTF]\label{specMTF}
	We seek $ \bm{\Lambda}_L \in \mathbb{X}_{M,L} $ such that the variational form:
	\begin{equation}
	( 	\textbf{\textsf{F}}(s)\bm{\Lambda}_L, \bm{\Phi}_L)_{\times} = (\bm{G}, \bm{\Phi}_L)_\times, \quad \text{for all }\bm{\Phi}_L\in \mathbb{Y}_{M,L},
	\end{equation}
	is satisfied for $ \bm{G} = (G_0, G_1, \ldots, G_M)\in \mathbb{V}_{\text{pw},M}. $
\end{problem}

\begin{remark}\label{rem:unsolved}
Existence and uniqueness of solutions for Problem \ref{specMTF} depend on the existence of an inf-sup condition. This result remains elusive by now, but various numerical experiments show the good behaviour of the proposed discretization scheme \cite{jerez2015local, jerez2016local}.
\end{remark}

\begin{remark}
	Notice that the discrete problem consists in finding a solution in the space $ \mathbb{V}_{pw, M} $ instead of $ \mathbb{V}_M $, which was the case for Problem \ref{MTF_freq}. Existence and uniqueness for the continuous problem remains valid in $ \mathbb{V}_{\text{pw}, M} $, as mentioned in \cite[Theorem 1]{jerez2015local}, due to the equivalence of the duality product between $ \mathbb{V}_M $ and $ \mathbb{V}_{\text{pw}, M} $ when test functions are elements of $ \dbtilde{\mathbb{V}}_{M} $.
\end{remark}

%
%
%

\subsection{Computation of Galerkin matrices}
We now compute the explicit integrals for each of the BIOs involved in the construction of the Calder\'on operator, as it was done in \cite{jerez2015local}. Integrals over $ \widehat{\Gamma}\times \widehat{\Gamma} $ are numerically approximated by the following two-step scheme:
\begin{itemize}
	\item[(a)] Computation of Chebyshev coefficients for the kernel in the inner integral.
	\item[(b)] Gauss-Legendre quadrature rule for the outer integral. 
\end{itemize}

For implementation purposes, we briefly explain the algorithm. Let $ K(\eta, \tau) \ : \ \widehat{\Gamma}\times \widehat{\Gamma}\rightarrow \mathbb{C} $ denote any of the BIO kernels.  For each subdomain $ \Omega_{i} $ and for a pair of interfaces $ \Gamma_{ji}$ and $ \Gamma_{i k} $, we proceed as follows:
\begin{enumerate}[i.]
	\item Set the number of Gauss-Legendre quadrature points $ \{\eta_j\}_{j=1}^{N_g} $ and Chebyshev points $ \{\tau_l\}_{l=1}^{N_c} $.
	\item Kernels $ K(\eta, \tau) $ are evaluated at each of these points and their values are stored. For the case of self-interactions, $ j=k $, the kernel is regularized with the Laplace kernel to extract the singularity (see Remark \ref{rem:reg}). 
	\item By means of the FFT over the array $ \{K(\eta_j, \tau_l) \}_{l=1}^{N_c}$, one
	computes approximations of the Chebyshev coefficients $ f_l(\eta_j)$,  $l = 1,\ldots, N_c,$ of the polynomial interpolant of the kernel 
	$$ K(\eta_j, \tau) \approx \displaystyle\sum_{l = 0}^{N_c} f_l(\eta_j)\upsilon_l(\tau), \quad \tau\in \widehat{\Gamma}, $$
	 for every Gauss-Legendre quadrature point $ \eta_{j}$, $j = 1, \ldots, N_g $.
	\item The outer integral is computed via a Gauss-Legendre quadrature rule.
\end{enumerate}

\begin{remark}
	\label{rem:reg}
	Regularization of the kernel is done for the case $ \Gamma_{ji} = \Gamma_{ki} $ by using the Laplace kernel $$ G(\nex, \ney) := \dfrac{1}{2\pi}\log \|\nex -\ney\|, \quad \nex \neq \ney, \quad  \nex,\ney\in\widehat{\Gamma}. $$
	Over the canonical segment $\widehat{\Gamma}$, the following kernel expansion \cite[Remark~4.2]{JNU2018} based on Chebyshev polynomials of the first kind, $ T_\ell(\eta):=\cos(\ell \theta) $, simplifies our computations:
	$$ \dfrac{1}{2\pi}\log|\eta-\tau| = \dfrac{1}{2\pi} \log 2 + \displaystyle\sum_{n=1}^{\infty}\dfrac{1}{n\pi}T_n(\eta)T_n(\tau), \quad \tau\neq \eta \in \widehat{\Gamma}.$$
	Specifically, two integrals have to be computed: one with the regularized kernel $ G_{i}(\nex, \ney; s) - G(\nex, \ney)$
	and one with the Laplace kernel that can be computed exactly. 	
\end{remark}

\subsubsection{Weakly Singular BIO}\label{sec:SLop}
For the weakly singular operator $ \V_{i}(s), \ s\in \mathbb{C}_+ $ defined on the boundary $\Gamma_i=\partial\Omega_{i} $, we need to compute integrals of the form:
\begin{equation}\label{eq:SLop}
\displaystyle\int\limits_{\widehat{\Gamma}} \displaystyle\int\limits_{\widehat{\Gamma}} G_{i}(\norm{\boldsymbol{h}_{i j}(\tau) - \boldsymbol{h}_{i k}(\eta)}; s)\upsilon_l(t)\upsilon_m(\tau)\sqrt{1-\tau^2} \text{d}\tau \text{d}\eta.
\end{equation}
For $ \eta \neq \tau $, the approximation
\begin{equation}
 G_{i}(\norm{\boldsymbol{h}_{i j}(\tau) - \boldsymbol{h}_{i k}(\eta)}; s) \approx \sum_{l = 0}^{N_c}f_l(\eta)\upsilon_l(\tau)
\end{equation}
holds and, by the orthogonality property \eqref{eq:orthogonal}, one obtains
\begin{equation}
\displaystyle\int\limits_{\widehat{\Gamma}} \displaystyle\int\limits_{\widehat{\Gamma}} G_{i}(\norm{\boldsymbol{h}_{i j}(\tau) - \boldsymbol{h}_{i k}(\eta)}; s)\upsilon_l(\eta)\upsilon_m(\tau)\sqrt{1-\tau^2} \text{d}\tau \text{d}\eta \approx \dfrac{\pi}{2}\displaystyle\int_{-1}^{1}f_l(\eta)\upsilon_l(\eta)\text{d}\eta,
\end{equation}
which is then approximated by a Gauss-Legendre quadrature rule. 
\subsubsection{Double Layer BIOs}
For the double layer BIOs and its adjoint, $ \K_{i}(s) $ and $ \K'_{i}(s)$, respectively, the approach is the same as in Section \ref{sec:SLop}. The only difference is that the Chebyshev coefficients have to be computed for the corresponding kernels.

\subsubsection{Hypersingular BIO}
For the hypersingular operator $ \W_{i}(s) $ defined on the boundary $ \partial\Omega_{i} $ we employ the following expression from \cite[Lemma~6.13, Theorem~6.15]{steinbach2007numerical}. Let $ \Gamma $ be an open measurable part of $ \partial\Omega_{i} $ and $ f $ and $ g $ be continuously differentiable on $ \Gamma.$ Then, it holds that
\begin{equation}\label{eq:Wop}
\begin{array}{lcl}
\langle \W_{i}(s)f, g\rangle_{\Gamma} &=& \langle  \V_{i}(s) \text{curl }f, \text{curl }g\rangle_{\Gamma}
 +\left( \dfrac{s}{c_{i}} \right)^2 \langle  \V_{i} (\bm{n}^{i}_{\nex}\cdot \bm{n}^{i}_{\ney} f), g\rangle_{\Gamma} \\
 &-& f(\nex)\left. \displaystyle\int\limits_{\Gamma} G_{i}(\norm{\nex - \ney}; s) \text{curl }g(\ney)\text{d}\Gamma_{\ney} \right|_{\nex\in \partial \Gamma},
\end{array}
\end{equation}
where $ \partial \Gamma $ denotes the endpoints of the open curve $ \Gamma. $ The first and second terms are computed as in the case for the weakly singular BIO and employing Chebyshev polynomials derivatives. The third term involves computing Chebyshev coefficients for which no quadrature rule is required.  

\section{Convolution Quadrature Schemes}
\label{sec:CQ}
Multistep-based and multistage CQ methods were originally introduced by Lubich in \cite{lubich1988convolution, lubich1988convolution_II} and in \cite{lubichRK}, respectively. For the sake of completeness, both methods will be explained along with their assumptions and limitations following \cite{hassell2016convolution}. 

Consider two functions $ f  :  [0, \infty)\rightarrow \mathcal{X} $ and $ g :  [0, \infty) \rightarrow \mathcal{Y}$ where $ \mathcal{X} $ and $ \mathcal{Y} $ are normed spaces. We further assume that $ F(s) := (\mathcal{L}f)(s) $ is known. We are interested in computing the convolution, which corresponds to

\begin{equation}\label{eq:CQ_conv}
h(t)  :=  \displaystyle \int_{0}^{t}f(\tau)g(t-\tau)\text{d}\tau
= \dfrac{1}{2\pi \imath} \displaystyle\int_{\sigma -i\infty}^{\sigma+i\infty} F(s) \left(   \int_{0}^{t}e^{s\tau}g(t-\tau)\text{d}\tau \right)\text{d}s.
\end{equation}
The inner integral in the right-hand side corresponds to the exact solution of the ordinary differential equation:
\begin{equation}\label{eq:ode}
\dot{y}(t)  = sy(t) + g(t) \quad t \in \mathbb{R}_+, \quad y(0) = 0,
\end{equation}
which can be solved by means of a multistep or a multistage linear method.

\subsection{Multistep Convolution Quadrature}\label{sec:multistep}
A multistep method \cite[Ch.~III.2]{wanner1993solving} for solving equation \eqref{eq:ode} is defined by parameters $ \alpha_{k}$, $\beta_{k}$, $k= 0, \ldots, m $ and a timestep $ \Delta t > 0$ with discrete times $ t_n = n\Delta t $ such that the new unknown corresponds to $ y_n \approx y(t_n)$, for  $n = 0, 1, \ldots $, in

\begin{equation}
\sum_{k = 0}^{m} \alpha_{k}y_{n+k-m}  =  \Delta t\sum_{k = 0}^{m}\beta_{k} (sy_{n+k-m} + g_{n+k-m}), \quad n = 0, 1, \ldots.
\end{equation}
At discrete times $ t_n $, an approximation for the convolution \eqref{eq:CQ_conv} is cast as  the following closed contour integral:
\begin{equation}\label{eq:CQ_convolution}
\begin{array}{rcl} 
h(t_n) &\approx& \dfrac{1}{2\pi \imath }\displaystyle\int_{\mathcal{C}} \dfrac{1}{\zeta^{n+1}}F\left(\frac{\delta(\zeta)}{\Delta t}\right)G(\zeta) \text{d}\zeta = \displaystyle \sum_{k = 0}^{n}\omega_k^F(\Delta t) g_{n-m},
\end{array}
\end{equation}
where 
\begin{equation}\label{eq:cq_weights}
 \delta(\zeta) = \dfrac{ \sum_{k = 0}^{m} \alpha_{m-k}\zeta^{k}}{ \sum_{k = 0}^{m} \beta_{m-k}\zeta^{k}},\quad F\left(\frac{\delta(\zeta)}{\Delta t}\right) = \displaystyle \sum_{n = 0}^{\infty}\omega_n^F(\Delta t)\zeta^n\quad \text{ and } \quad G(\zeta) = \displaystyle\sum_{n= 0}^{\infty}g_n \zeta^n.
 \end{equation}
For implementation purposes, a good choice for the integration contour $ \mathcal{C} $ is a circle of radius $ 0<\lambda < 1 $, explicitly $ \lambda = \varepsilon^{\frac{1}{2N}} $ \cite{banjai2008rapid},  where $\varepsilon $ refers to machine precision and $ N $ is the number of quadrature points used. Later on, this number will coincide with the number of timesteps for the time discretization. Indeed, this allows the stable use of FFT to compute the integral by means of a trapezoidal rule. Following \cite[Sections~3.2-3.3]{hassell2016convolution}, the final expression for \eqref{eq:CQ_convolution} becomes
\begin{equation}\label{eq:trapezoid}
h(t_n) =\dfrac{1}{2\pi \imath }\displaystyle\int_{\mathcal{C}} \dfrac{1}{\zeta^{n+1}}F\left(\frac{\delta(\zeta)}{\Delta t}\right)G(\zeta) \text{d}\zeta \approx \lambda^{-n} \left(\dfrac{1}{N+1} \displaystyle \sum_{k  = 0}^{N} \widehat{F}_{k}   \left( \sum_{j = 0}^{N}\lambda^j g_j \zeta_{N+1}^{-jk}   \right)\zeta_{N+1}^{kn} \right), 
\end{equation}
where $ \widehat{F}_{k}  = F\left(\dfrac{\delta(\lambda \zeta_{N+1}^{-k})}{\Delta t}\right) $, $ \zeta_{N+1} = e^{\frac{2\pi \imath }{N+1}} $. Moreover, the method converges at the rate of the multistep method chosen. Due to Dahlquist's barrier theorem \cite{wanner1996solving}, A-stable multistep methods are limited to order less than or equal to two, which is the main disadvantage of multistep-based CQ.

\begin{remark}
The expression \eqref{eq:trapezoid} is well defined as long as $ F, G $ and $ \delta$ are analytic. The contour  $ \mathcal{C} $ must belong to the analyticity region of the transfer function $ F\left(\frac{\delta(\zeta)}{\Delta t}\right)G(\zeta)$ and wind around zero once, which is guaranteed by using A-stable multistep methods \cite{lubich1988convolution}. 	
\end{remark}

\subsection{Multistage Convolution Quadrature}
\label{sec:multistage}
We solve \eqref{eq:ode} by employing A-stable implicit multistage methods of arbitrary order \cite{banjai2011runge} instead of the more restrictive multistep methods. 

Let $ A\in\mathbb{R}^{m\times m}, \ \bm{b}, \bm{d} \in \mathbb{R}^m $ be the Butcher's tableau for a given $m$-stage Runge-Kutta method \cite[Chapter 4]{wanner1996solving}, and defining $ y_{nj} \approx y(t_n + d_j\Delta t) $, for \eqref{eq:ode}, the problem consists in looking for a vector-valued function $ \bm{y}_n := (y_{n1}, \ldots, y_{nm}) = (y(t_n + d_1\Delta t), \ldots, y(t_n + d_m\Delta t)) = y(t_n + \bm{d}\Delta t), \ n = 0, \ldots, N,$ of stage solutions such that

\begin{equation}\label{eq:RK}
\begin{array}{rclll}
\bm{y}_{n} &=& y_n \bm{1} &+& \Delta t \  A(\bm{y}_n + g(t_n + \bm{d}\Delta t)), \\
&&&&\\
y_{n+1} &=& y_n &+& \Delta t \  \bm{b} \cdot (\bm{y}_n + g(t_n + \bm{d}\Delta t)). 
\end{array}
\end{equation}
Assume that $\bm{G}(\zeta) = \sum_{n = 0}^{\infty} g(t_n +\bm{d}\Delta t) \zeta^n$ and $  \delta^{\text{RK}}(\zeta) :=  \left( \bm{1}\bm{b}^T  \dfrac{\zeta}{1-\zeta} + A \right)^{-1}$. Following a similar procedure to that of the multistep case \cite{banjai2011runge}, one finds
\begin{equation}\label{eq:RKCQ_convolution}
 \bm{h}_n = h(t_n + \bm{d}\Delta t) \approx \dfrac{1}{2\pi \imath }\displaystyle\int_{\mathcal{C}} \dfrac{1}{\zeta^{n+1}} F\left(\dfrac{\delta^{\text{RK}}(\zeta)}{\Delta t}\right)\bm{G}(\zeta) \text{d}\zeta.  
\end{equation}
If we consider stiffly accurate Runge-Kutta methods --for example RadauIIA or LobattoIIIC classes \cite{lubichRK}-- we have the relation $h_{nm} = h_{n+1}$. The procedure to obtain the convolution at each discrete time follows the same idea employed for the approximation \eqref{eq:trapezoid} \cite[Section~5.5]{hassell2016convolution}.


\subsection{Convolutional Equations}\label{sec:conv_eq}
We have shown how to compute convolutions with CQ, but we are interested in solving equations where the unknown is one of the terms involved in the convolution.
As explained in Sections \ref{sec:multistep} and \ref{sec:multistage}, convolution can be approximated as \eqref{eq:CQ_convolution} or \eqref{eq:RKCQ_convolution}, for multistep or multistage strategies, respectively. Without loss of generality, we now follow the notation of multistep methods.

\begin{problem}[Semi-discrete CQ-MTF]
	\label{CQMTF}
	Let $ N \in \mathbb{N}$ be the fixed number of timesteps, $\Delta t = T/N>0$ be the timestep chosen for a CQ scheme, $ T >0 $ denotes the final time of computation and $t_m = m\Delta t$ discrete times for $m = 0,\ldots, N $, with $t_N=T$. Let $ \delta(\zeta) $ be given by the choice of a multistep or multistage A-stable method. We look for solutions $ \bm{\lambda}_m \in \mathbb{V}_{pw, M}$ at each $t_m$, $m = 0,\ldots, N $, such that
	
	\begin{equation}
	\left(\sum_{m = 0}^{n}\omega_m^{\mathbf{\mathsf{F}}}(\Delta t)\bm{\lambda}_{n-m}, \bm{\Phi}    \right)_{\times} = \left(\bm{g}(t_n), \bm{\Phi}\right)_{\times}, \quad \text{for all }\bm{\Phi}\in \dbtilde{\mathbb{V}}_{M},
	\end{equation}
	where $ \omega_m^{\mathbf{\mathsf{F}}}(\Delta t)$ are the CQ weights coming from the analytic expansion \eqref{eq:cq_weights}, $ \mathbf{\mathsf{F}}$ is the multiple traces operator defined in Problem \ref{MTF_freq} and $ \bm{g}(t)$ is the boundary data defined in Problem \ref{TDMTF}.
\end{problem}

As Problem \ref{CQMTF} corresponds to a triangular Toeplitz system, existence and uniqueness of solutions depend on the solvability of the diagonal terms
\begin{equation}\label{eq:w0}
\left( \omega_0^{\mathbf{\mathsf{F}}}(\Delta t)\bm{\lambda}, \bm{\varphi}\right)_{\times} = (\bm{\Xi}, \bm{\Phi})_{\times}, \quad \text{for all }\bm{\Phi}\in \dbtilde{\mathbb{V}}_M,
\end{equation}
for a given $ \bm{\Xi}\in \mathbb{V}_{\text{pw},M}$. To this end, we notice that convolution weights satisfy \cite{hassell2016convolution}
$$\omega_n^{\mathbf{\mathsf{F}}}(\Delta t) = \dfrac{1}{n!}\dfrac{\text{d}^n}{\text{d}\zeta^n} \left. \left( \mathbf{\mathsf{F}}\left(  \dfrac{\delta(\zeta)}{\Delta t} \right)   \right)\right|_{\zeta = 0}, \ n = 0,\ldots, N, $$
and so $ \omega_0^{\mathbf{\mathsf{F}}}(\Delta t) = \mathbf{\mathsf{F}}\left(  \dfrac{\delta(0)}{\Delta t} \right)   $, for which existence and uniqueness of solutions for \eqref{eq:w0} follows by arguments similar to \cite[Theorem~11]{hiptmair2012multiple}.

\begin{theorem}\label{proofCQMTF} There exists a unique solution $ \{\bm{\lambda}_n \}_{n = 0}^N \in \left[ \mathbb{V}_{\text{pw},M}  \right]^{N+1}
 $ for Problem \ref{CQMTF}.
\end{theorem}
\begin{proof}
	Problem \ref{CQMTF} can be written as a triangular Toeplitz system:
	\begin{equation*}
	\begin{array}{ccccccccl}
	\left( \omega_0^{\mathbf{\mathsf{F}}}(\Delta t)\bm{\lambda}_0, \bm{\varphi}\right)_{\times} 
	&&\ldots
	&& 0
	&=& (\bm{g}(t_0), \bm{\varphi})_{\times}, \\
	\vdots
	&&\ddots
	&& \vdots
	&&\vdots
	&&\\
	\left( \omega_N^{\mathbf{\mathsf{F}}}(\Delta t)\bm{\lambda}_0, \bm{\varphi}\right)_{\times} 
	&+&
	\ldots
	&+&
	\left( \omega_0^{\mathbf{\mathsf{F}}}(\Delta t)\bm{\lambda}_N, \bm{\varphi}\right)_{\times}  
	&=& (\bm{g}(t_N), \bm{\varphi})_{\times},\\
	\end{array}
	\end{equation*}
	for all $ \bm{\varphi}\in \dbtilde{\mathbb{V}}_M. $ As \eqref{eq:w0} has a unique solution in $ \mathbb{V}_{\text{pw}, M} $ for every right-hand side $ \bm{\Xi}\in \mathbb{V}_{\text{pw},M}$, each one of the previous equations is uniquely solvable by $ \bm{\lambda}_n$, $n= 0,\ldots, N $, by Theorem \ref{MTF_freq_exist_uniq}.
	\end{proof}
Finally, recalling the notation introduced in Section \ref{sec:spectral}, the fully discrete problem reads as follows.
\begin{problem}[Fully Discrete CQ-Spectral Galerkin MTF]
	\label{discreteTDMTF}
	Let $ N \in \mathbb{N}, \ \Delta t = T/N >0 $ be the timestep chosen for a CQ scheme, $ T >0 $ be the final time, $n\in\mathbb{N}$ so that $ t_0 = 0, \ldots, t_n  = n\Delta t, \ldots, t_N = T $ discrete times. Let $ \delta(\zeta) $ be given by the choice of a multistep or multistage A-stable method. We look for solutions $ \bm{\lambda}_n \in \mathbb{X}_{M,L}$ at each discrete times $t_n$, $ n = 0,\ldots, N$, such that
	\begin{equation}
	\left(\sum_{m = 0}^{n}\omega_m^{\mathbf{\mathsf{F}}}(\Delta t)\bm{\lambda}_{n-m}, \bm{\varphi}    \right)_{\times} = \left(\bm{g}(t_n), \bm{\varphi}\right)_{\times}, \quad \text{for all }\bm{\varphi}\in \mathbb{Y}_{M,L},
	\end{equation}
	where $ \omega_m^{\mathbf{\mathsf{F}}}(\Delta t), \ m = 0,\ldots, N, $ are the CQ weights coming from the analytic expansion \eqref{eq:cq_weights}, $ \mathbf{\mathsf{F}}$ is the multiple traces operator defined in Problem \ref{MTF_freq} and $ \bm{g}(t)$ is the boundary data defined in Problem \ref{TDMTF}.
\end{problem}

\begin{remark}
Well-posedness of Problem \ref{discreteTDMTF} remains as an open problem, as it depends on the existence and uniqueness of solutions for the non-conforming spatial spectral scheme (see Remark \ref{rem:unsolved}). 
\end{remark}

\section{Numerical Experiments}
\label{sec:numexp}
We present several numerical experiments to validate our proposed CQ-MTF discretization for different scenarios. All computations were performed on {\sc{Matlab}} 2018a, 64bit, running on a GNU/Linux desktop machine with a 3.80 GHz CPU and 32GB RAM.\footnote{The code is available in \url{http://www.github.com/ijlabarca/cqmtf}.} We measure different error norms to verify that our implementation is correct and analyze its performance. We use equivalent norms for the spaces $ H^{\pm \half}(\Gamma) $ based on the single layer operator and its inverse, for a wavenumber $ k = 10\imath $ ($ s = 10 $ for modified Helmholtz equation).

For each numerical example, we compute relative trace errors over each interface for a density $ \varphi\in H^{s}(\Gamma_{i}) $ compared to a reference density $ \varphi^{\text{ref}}\in H^{s}(\Gamma_{i}) $, with $ s \in \{-\half, \half \} $. We express it as follows
\begin{equation}\label{eq:trace_error}
\begin{array}{rcll}
\texttt{TraceError}(\Gamma_{i}) &:=& \dfrac{\left( \displaystyle\sum_{n = 0}^{N}\norm{\varphi^{\text{ref}}_n - \varphi_n}^2_{H^{s}(\Gamma_{i})} \right)^{\half} }{  \left( \displaystyle\sum_{n = 0}^{N}\norm{\varphi^{\text{ref}}_n}^2_{H^{s}(\Gamma_{i})} \right)^{\half}}, &i = 0, 1, \ldots, M.\\
\end{array}
\end{equation}
We also compute field errors on a set of sample points $S$ in each domain $ \Omega_{i} $ by using the representation formula \eqref{eq:representation_formula}, compared to a reference field $ u^{\text{ref}} $

\begin{equation}\label{eq:errorU}
\texttt{ErrorU}(\Omega_{i}) :=  \dfrac{
	\left( 
	{\displaystyle\sum_{n = 0}^{N}}{\ \displaystyle\sum\limits_{\bm{x}\in S}} |u^{\text{ref}}_n(\bm{x}) - u_n(\bm{x})|^2 \right)^{\half}}{\ \left(\displaystyle\sum_{n = 0}^{N} \displaystyle\sum\limits_{\bm{x}\in S} |u^{\text{ref}}_n(\bm{x})|^2 \right)^{\half}}, \qquad S\in \Omega_{i}.
\end{equation}
CQ is implemented using the BDF2 multistep method \cite[Chapter 5]{wanner1996solving}, completely determined by the polynomial
\begin{equation}
\label{eq:BDF2_gamma}
\gamma(\zeta) = \dfrac{3}{2} - 2\zeta + \dfrac{1}{2}\zeta^2, 
\end{equation}
and multistage methods (Runge-Kutta CQ) corresponding to the two-stage RadauIIa quadrature \cite[Chapter 7]{wanner1996solving}, whose Butcher's tableau is defined by
\begin{equation}
\label{eq:Radau}
A := \begin{pmatrix}
5/12 & -1/12\\ 3/4 & 1/4
\end{pmatrix}, \quad \bm{b} := \begin{pmatrix}
3/4 \\ 1/4
\end{pmatrix} \quad \text{and }\quad  \bm{d} := \begin{pmatrix}
1/3 \\ 1
\end{pmatrix}, 
\end{equation}
and the three-stage LobattoIIIc quadrature  \cite[Chapter 7]{wanner1996solving}, whose Butcher's tableau is defined by
\begin{equation}
\label{eq:Lobatto}
A := \begin{pmatrix}
1/6 & -1/3 & 1/6\\ 1/6 & 5/12 & -1/12\\ 1/6 & 2/3 & 1/6
\end{pmatrix}, \quad \bm{b} := \begin{pmatrix}
1/6 \\ 2/3 \\ 1/6
\end{pmatrix} \quad \text{and }\quad  \bm{d} := \begin{pmatrix}
0 \\ \half \\ 1
\end{pmatrix}.
\end{equation}
Properties of the Radau IIA and Lobatto IIIC methods are summarized in Table \ref{tab:RK-methods}. It is important to notice the differences among these as they play an important role in the expected order of convergence in the numerical scheme. The best result that one can expect is the classical order of convergence $ p$, which is $2.0$ for the BDF2 method and depends on the number of stages for Runge-Kutta methods. Order reduction phenomena are also common and have been observed for other wave propagation problems and CQ applications \cite{banjai2011runge, banjai2012runge, rieder2017convolution}. In general, this depends on the stage order $ q $, which in our case is the same for RadauIIA and LobattoIIIC.  Recall that we do not provide numerical estimates and regularity results for solutions of wave propagation problems over composite materials. Thus, we cannot claim more precise bounds other than these lower and upper ones for error convergence in time domain.

\begin{table}[t]
	\centering 
	
	\begin{tabular}{|c|c|c|c|}
		\hline
		Method & Stages & Stage order ($ q $) & Classical order ($ p $) \\ \hline \hline 
		Radau IIA & $ s $ & $ s $ & $ 2s-1 $ \\ \hline 
		Lobatto IIIC & $ s $ & $ s-1 $ & $ 2s-2 $ \\ \hline 
	\end{tabular}
	\caption{Properties of Runge-Kutta methods used for CQ\cite[Chapter 7]{wanner1996solving}.}
	\label{tab:RK-methods}
\end{table}

\begin{figure}[b]
	\begin{subfigure}[c]{0.5\textwidth}
		\centering
		\includegraphics[width=0.85\linewidth]{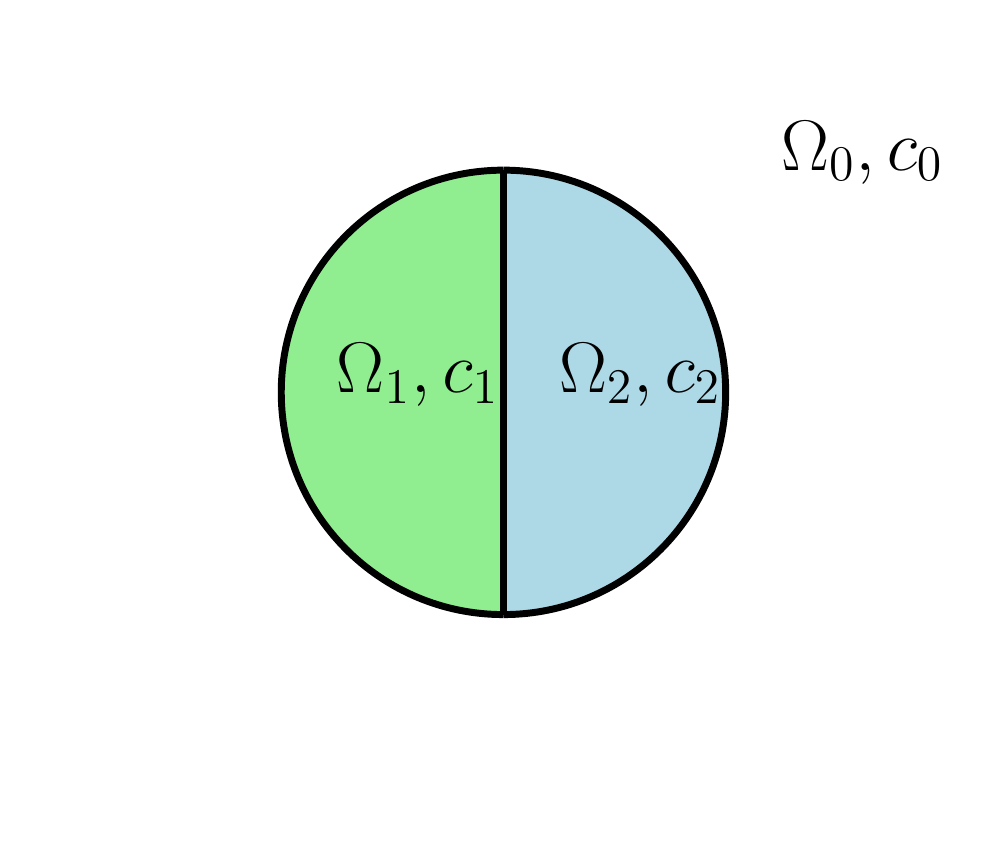}
		\caption{Circle with two subdomains.}
		\label{fig:domain_mtf1}
	\end{subfigure}
	\begin{subfigure}[c]{0.5\textwidth}
		\centering 
		\includegraphics[width=0.85\linewidth]{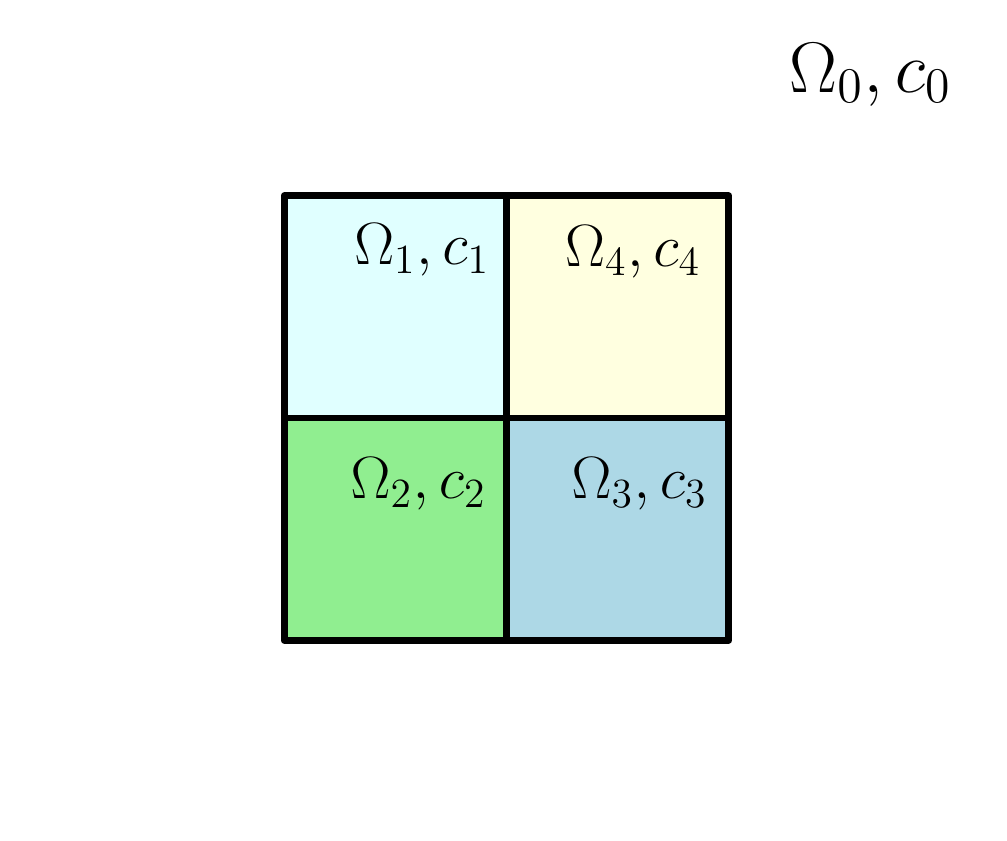}
		\caption{Square with four subdomains.}
		\label{fig:domain_mtf2}
	\end{subfigure}
	\caption{Geometries used for numerical experiments.}
	\label{fig:domains}
\end{figure}

\subsection{Spectral dicretization with Chebyshev polynomials}\label{sec:test_spectral}
First, we show convergence results for the non-conforming spectral discretization presented in Section \ref{sec:spectral}. Although we are using a spectral non-conforming discretization with Chebyshev polynomials, convergence rates are not expected to be exponential. This is related to the regularity of solutions of Helmholtz transmission problems in non-smooth domains \cite{grisvard2011elliptic}. Our aim is to show that we still obtain high-order convergence rates with accurate solutions using only a relatively small number of degrees of freedom.

We solve Problem \ref{MTF_freq} for a domain given by the circle with radius $ r = 0.5 $ with two subdomains (see Figure \ref{fig:domain_mtf1}). The volume problem to be solved by the MTF in these examples corresponds to
\begin{equation}
\left\{ 
\begin{array}{rcll}
-\Delta U + s_{i}^2 U &=& 0 & \text{in } \Omega_{i}, \ i = 0,1,2,\\
\jump{\gamma U}_{0i} &=& -\gamma U^{\text{inc}} |_{\Gamma_{0i}}, &\quad \qquad  i = 1, 2, \\
\jump{\gamma U}_{12} & = & 0,&\\
\jump{\gamma U}_{21} & = & 0,&
\end{array}\right. 
\end{equation}

\begin{table}[t]
	\centering 
	\caption{Parameters used in Section \ref{sec:test_spectral}. Convergence results are shown in Figure \ref{fig:errors_freq}.}
	\begin{tabular}{|l||c|c|c|c|c|}
		\hline\hline 
		&$ s_0 $ & $ s_1 $ & $ s_2 $ & $ U^{\text{inc}} $ & $ \bm{d} $ \\ \hline 
		\text{Example A (blue)}& $ -i $ & $ -2i $ & $ -4i $ & $ \exp(-s_0 \bm{x}\cdot \bm{d}) $ & $ (0, -1) $\\ \hline 
		\text{Example B (green)}& $ 1-i $ & $ 2-2i $ & $ 4-4i $ & $ \exp(-s_0 \bm{x}\cdot \bm{d}) $ & $ (0, -1) $\\\hline 
		\text{Example C (brown)}& $ 1-i $ & $ 10-10i $ & $ 20-20i $ & $ \exp(-s_0 \bm{x}\cdot \bm{d}) $ & $ (0, -1) $\\\hline 
		\text{Example D (purple)}& $ 1-i $ & $ 10-10i $ & $ 100-100i $ & $ \exp(-s_0 \bm{x}\cdot \bm{d}) $ & $ (0, -1) $\\\hline 
	\end{tabular}
	\label{tab:params_test_spectral}
\end{table}

The parameters employed are shown in Table \ref{tab:params_test_spectral}. Errors are measured with respect to a highly resolved solution. Convergence results are displayed in Figure \ref{fig:errors_freq}. Different values of $ s $ are used in order to study errors for different problems. As expected, Galerkin solutions initially converge spectrally and then reach a finite order convergence rate due to the non-smoothness of the domains. Small errors are obtained with few degrees of freedom, which is adequate for coupling with BDF2 and two-stage RadauIIA CQ methods. For LobattoIIIc, a higher number of degrees of freedom is required to match its high order convergence in time. In order to study convergence of CQ time-stepping, we fix a high polynomial degree to suppress errors due to the Laplace domain solver. 

\begin{figure}[t]
	\hspace{-0.7cm}
	\begin{subfigure}[c]{0.35\textwidth}
		\centering 
		\includegraphics[width=1.0\linewidth]{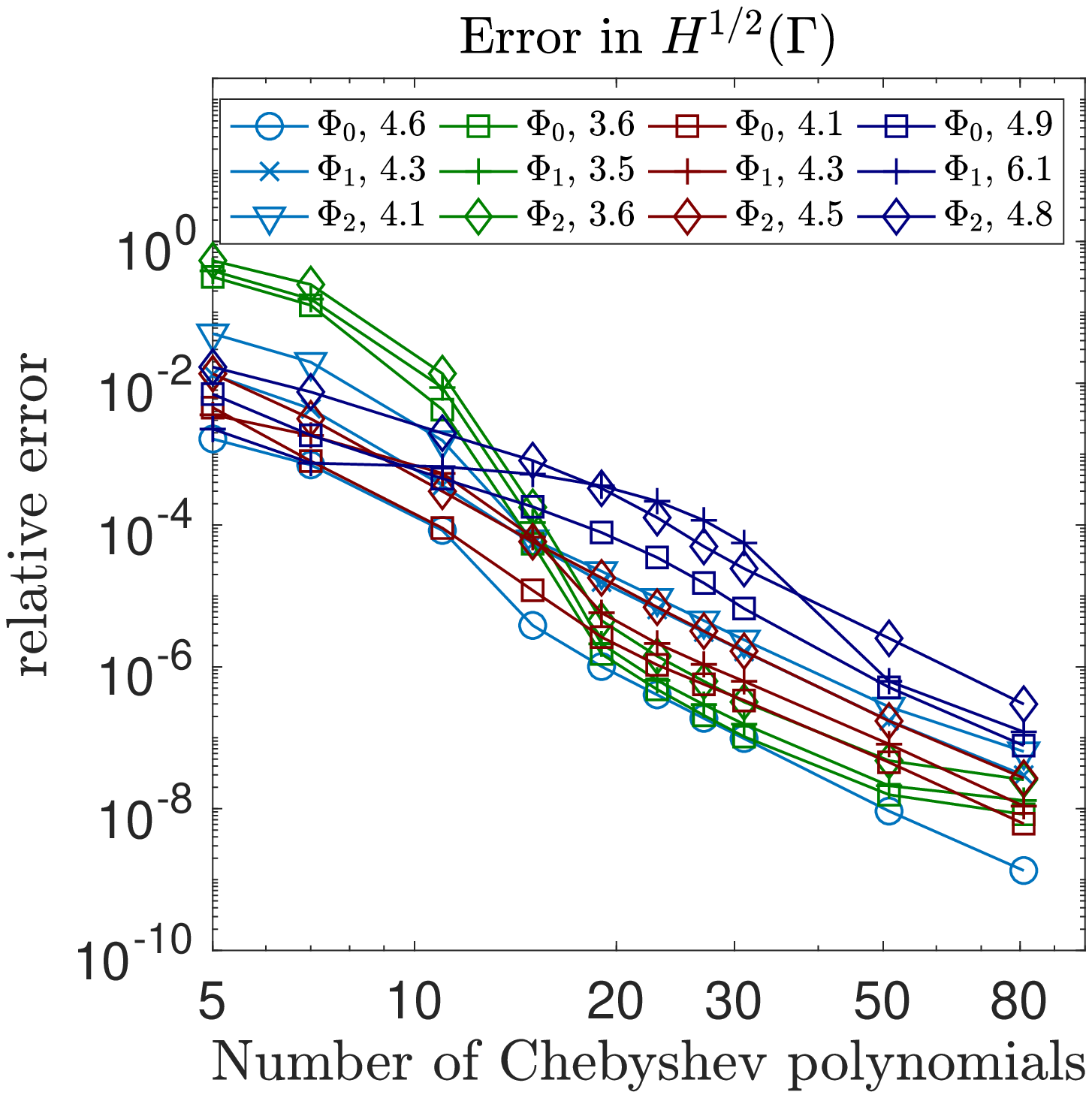}	
		\caption{Dirichlet trace errors} \label{fig:error_dirichlet_sp}
	\end{subfigure}
	\hspace{-0.3cm}
	\begin{subfigure}[c]{0.35\textwidth}
		\centering 
		\includegraphics[width=1.0\linewidth]{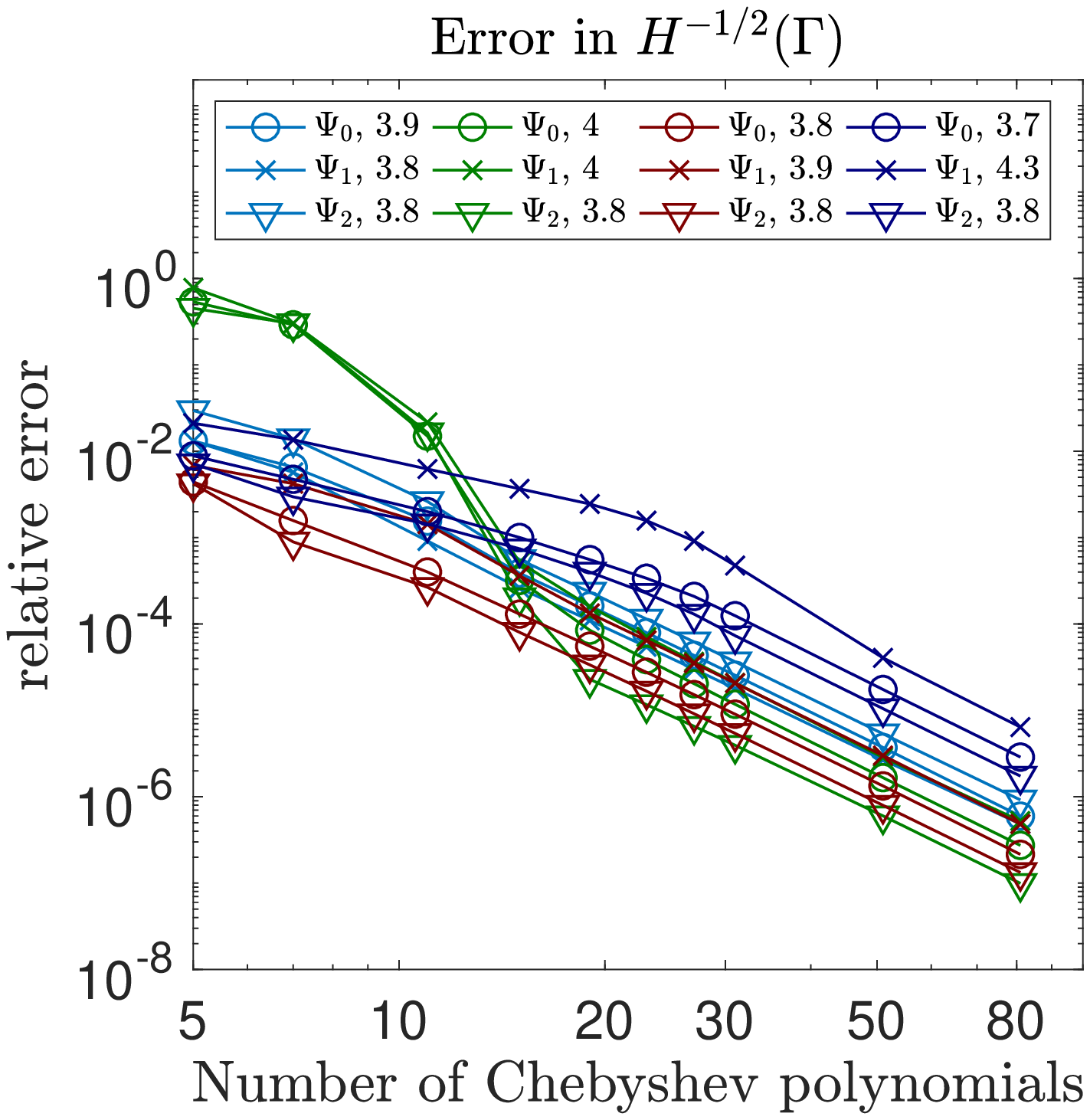}
		\caption{Neumann trace errors} \label{fig:error_neumann_sp}
	\end{subfigure}
	\hspace{-0.3cm}
	\begin{subfigure}[c]{0.35\textwidth}
		\centering 
		\includegraphics[width=1.0\linewidth]{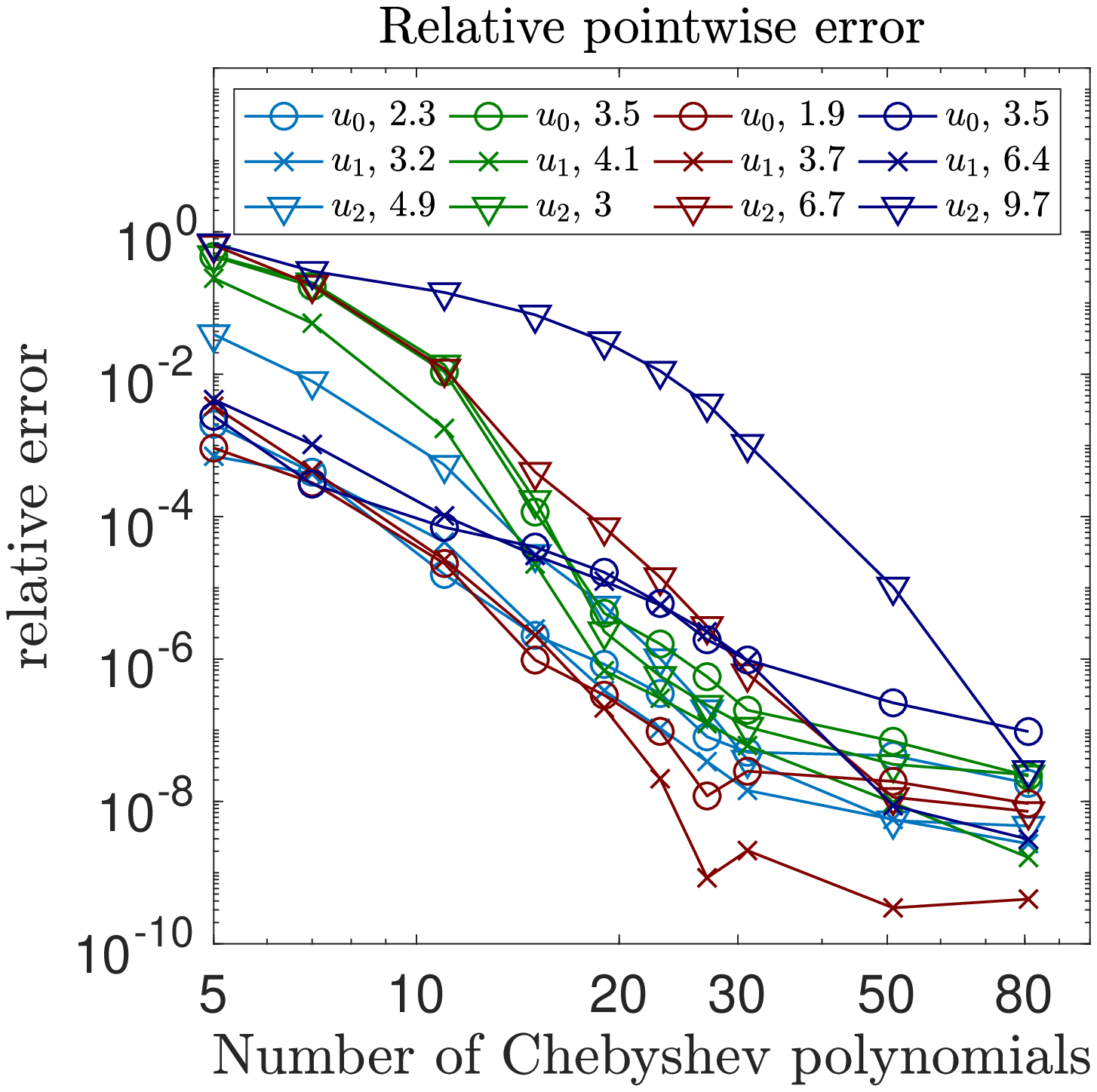}	
		\caption{Field errors} \label{fig:error_U_sp}
	\end{subfigure}
	\caption{Different relative error norms for the two subdomains case (see Figure \ref{fig:domain_mtf1})
 measured with respect to a highly resolved solution. Legends show estimated orders of convergence obtained by means of least squares fittings.}
	\label{fig:errors_freq}
\end{figure}

\subsection{Manufactured time-domain solutions with no triple points}\label{sec:test0}
We now analyze our time-domain CQ-MTF by considering the case of a single domain, i.e. no triple points. The domain $ \Omega $ consists of a circle of radius $ r = 0.5 $. We construct transmission conditions such that the exterior solution is zero whereas the interior one corresponds to
\begin{equation}\label{eq:manufactured0}
u_{1}(\bm{x}, t) := f(c_{1}t - t_{lag} - \bm{x}\cdot \bm{d}), \quad f(t) := \sin(\omega t)\eta(t, 0.2, 2), \quad i = 1,2,
\end{equation}
\begin{table}[t]
	\centering 
	
	\begin{tabular}{|c|c|c|c|c|c|c|}
		\hline\hline 
		$ c_0 $ & $ c_1 $ &$  \omega $ & $ T $ &$ t_{\text{lag}} $ &$ \bm{d} $& $ N_{\text{cheb}} $ \\ \hline 
		$ 1$ & $ 0.5 $ & $ 1 $& $ 5 $ & $ 0.5 $&$ (1, 0) $& $ 40 $\\ \hline 
	\end{tabular}
	\caption{Parameters used in Section \ref{sec:test0}.}
	\label{tab:params_test0}
\end{table} where the parameters used are provided in Table \ref{tab:params_test0}. Therein, $ \eta $ is a smooth version of the Heaviside function defined as
\begin{equation}
\eta (t, t_0, t_1) := \left\{   
\begin{array}{cl}
0 & 0 \leq t < t_0, \\ 
1-\exp\left(\dfrac{2e^{-1/\tau}}{\tau-1}\right), \ \tau = \dfrac{t-t_0}{t_1-t_0}, & t_0 < t < t_1, \\
1 & t > t_1. \\
\end{array}
\right.
\end{equation}

Convergence results for BDF2 and Runge-Kutta CQ methods are displayed in Figure \ref{fig:analytic_sol-simple}. Classical orders of convergence are restricted to the exterior domain (Fig. \ref{fig:error_U_0}), whereas for interior ones, CQ methods suffer from reduced convergence rates (Fig. \ref{fig:error_dirichlet_0} \ref{fig:error_neumann_0}). We observe that the three implemented methods share the same convergence rates for Neumann traces. This is explained by the low temporal regularity of the solutions, as normally the methods would achieve different orders of convergence.

\begin{figure}[t]
	\hspace{-0.7cm}
	\begin{subfigure}[c]{0.35\textwidth}
		\centering 
		\includegraphics[width=1.0\linewidth]{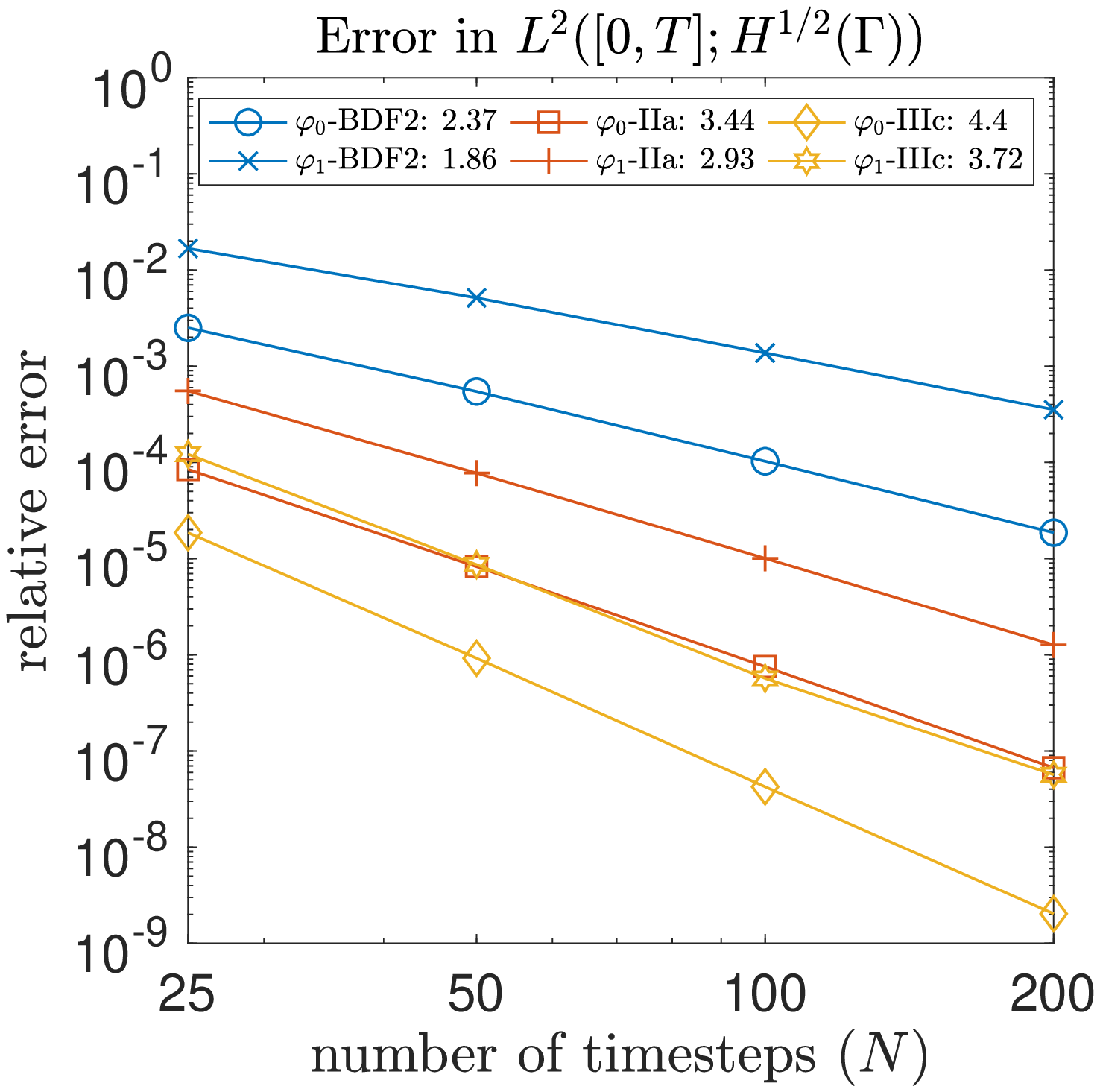}	
		\caption{Dirichlet trace errors} \label{fig:error_dirichlet_0}
	\end{subfigure}
	\hspace{-0.3cm}
	\begin{subfigure}[c]{0.35\textwidth}
		\centering 
		\includegraphics[width=1.0\linewidth]{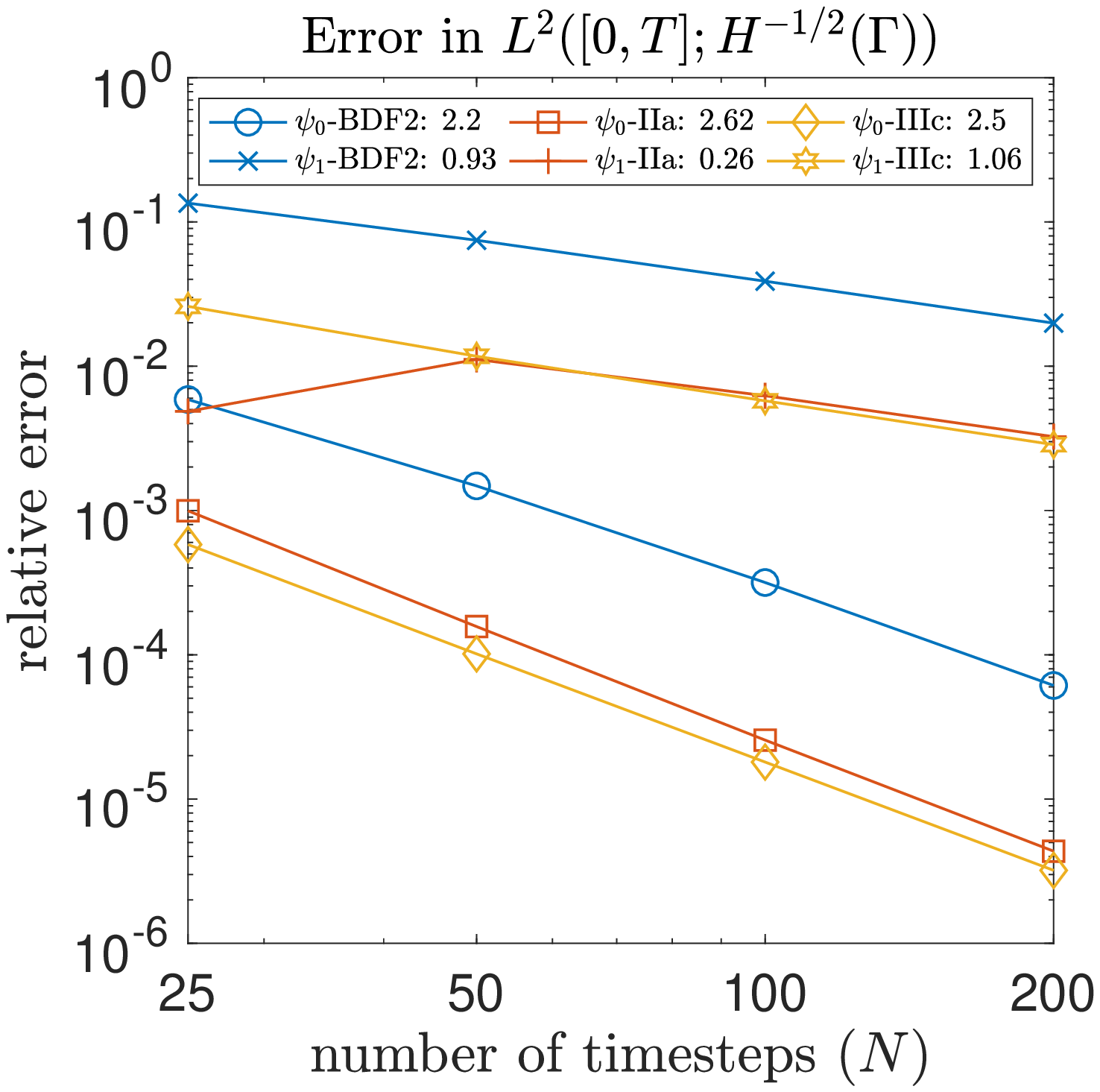}
		\caption{Neumann trace errors} \label{fig:error_neumann_0}
	\end{subfigure}
	\hspace{-0.3cm}
	\begin{subfigure}[c]{0.35\textwidth}
		\centering 
		\includegraphics[width=1.0\linewidth]{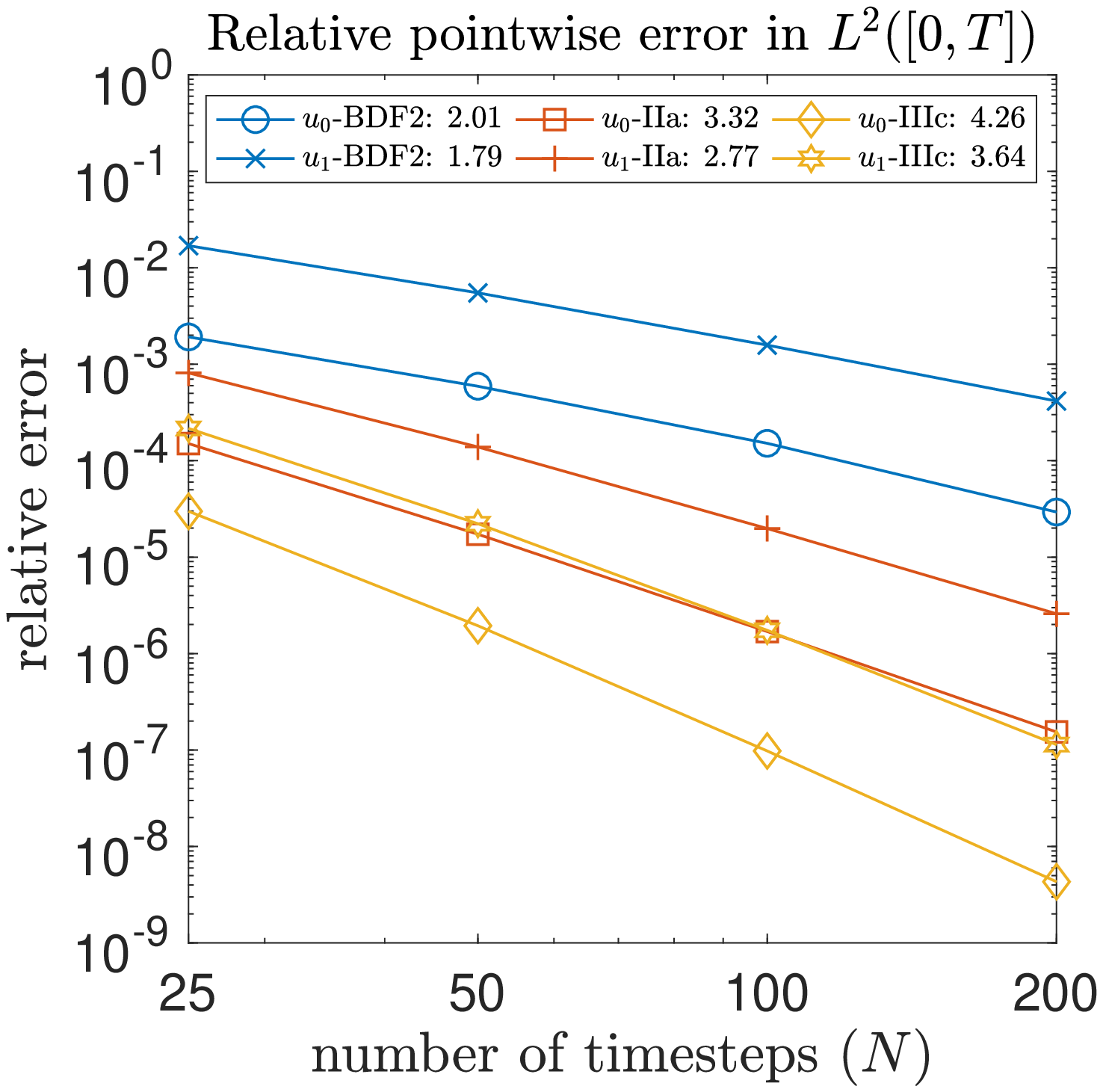}	
		\caption{Field errors} \label{fig:error_U_0}
	\end{subfigure}
	\caption{Relative errors for the transmission problem measured with respect to the analytical solution. Errors are computed as explained in \eqref{eq:trace_error} and \eqref{eq:errorU}. Legends show the estimated order of convergence obtained by means of least-squares fittings.}
	\label{fig:analytic_sol-simple}
\end{figure}

\subsection{Manufactured time-domain solutions with artificial subdomains}\label{sec:test1}

We now validate the use of CQ-MTF by considering the case of artificial subdomains. The domain $ \Omega $ consists in a circle of radius $ r = 0.5 $ separated into two subdomains $ \Omega_1 $ and $ \Omega_2 $ --left and right semicircles, respectively in Figure \ref{fig:domain_mtf1}. We use the same manufactured solutions from Section \ref{sec:test0}. As the physical parameters of both subdomains are identical, the only difference from the case in Section \ref{sec:test0} is the presence of an artificial triple point.

\begin{table}[t]
	\centering 
	\begin{tabular}{|c|c|c|c|c|c|c|c|}
		\hline\hline 
		$ c_0 $ & $ c_1 $ & $ c_2 $ &$  \omega $ & $ T $ &$ t_{\text{lag}} $ &$ \bm{d} $& $ N_{\text{cheb}} $ \\ \hline 
		$ 1$ & $ 0.5 $ &$ 0.5 $& $ 1 $& $ 5 $ & $ 0.5 $&$ (1, 0) $& $ 40 $\\ \hline 
	\end{tabular}
		\caption{Parameters used in Section \ref{sec:test1}.}
	\label{tab:params_test1}
\end{table}

Error convergence results for BDF2 and Runge-Kutta-based CQ (RadauIIa and LobattoIIIc) are shown in Figure \ref{fig:analytic_sol}. Only small differences can be pointed out with respect to the results of Section \ref{sec:test0}, in particular, for the convergence of Neumann traces as they appear to improve for RadauIIa. However, these results are not sufficient to state that increasing the number of subdomains artificially impacts the convergence of the method.

\begin{figure}[t]
	\hspace{-0.7cm}
	\begin{subfigure}[c]{0.35\textwidth}
		\centering 
		\includegraphics[width=1.0\linewidth]{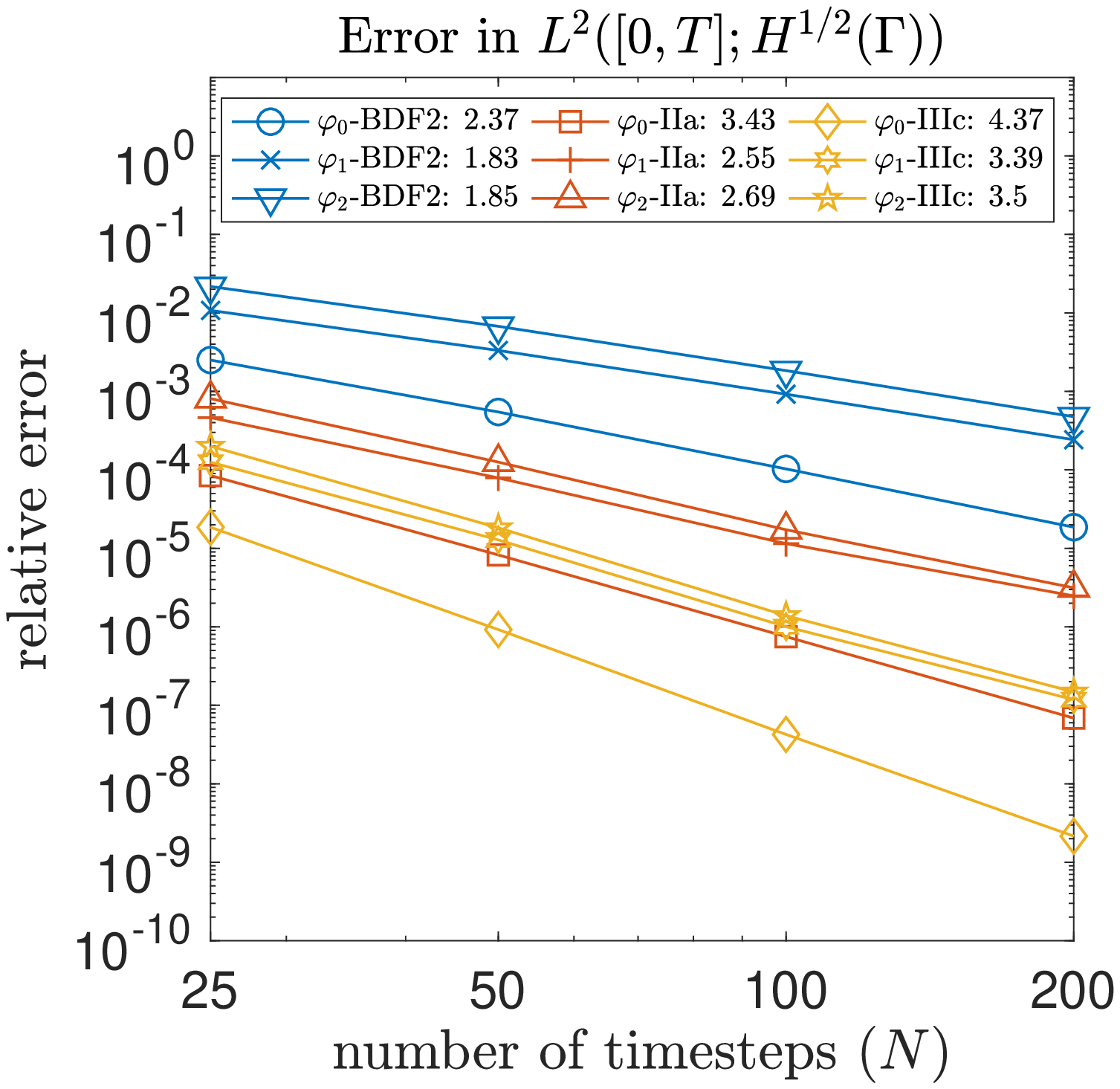}	
		\caption{Dirichlet trace errors} \label{fig:error_dirichlet_1}
	\end{subfigure}
	\hspace{-0.3cm}
	\begin{subfigure}[c]{0.35\textwidth}
		\centering 
		\includegraphics[width=1.0\linewidth]{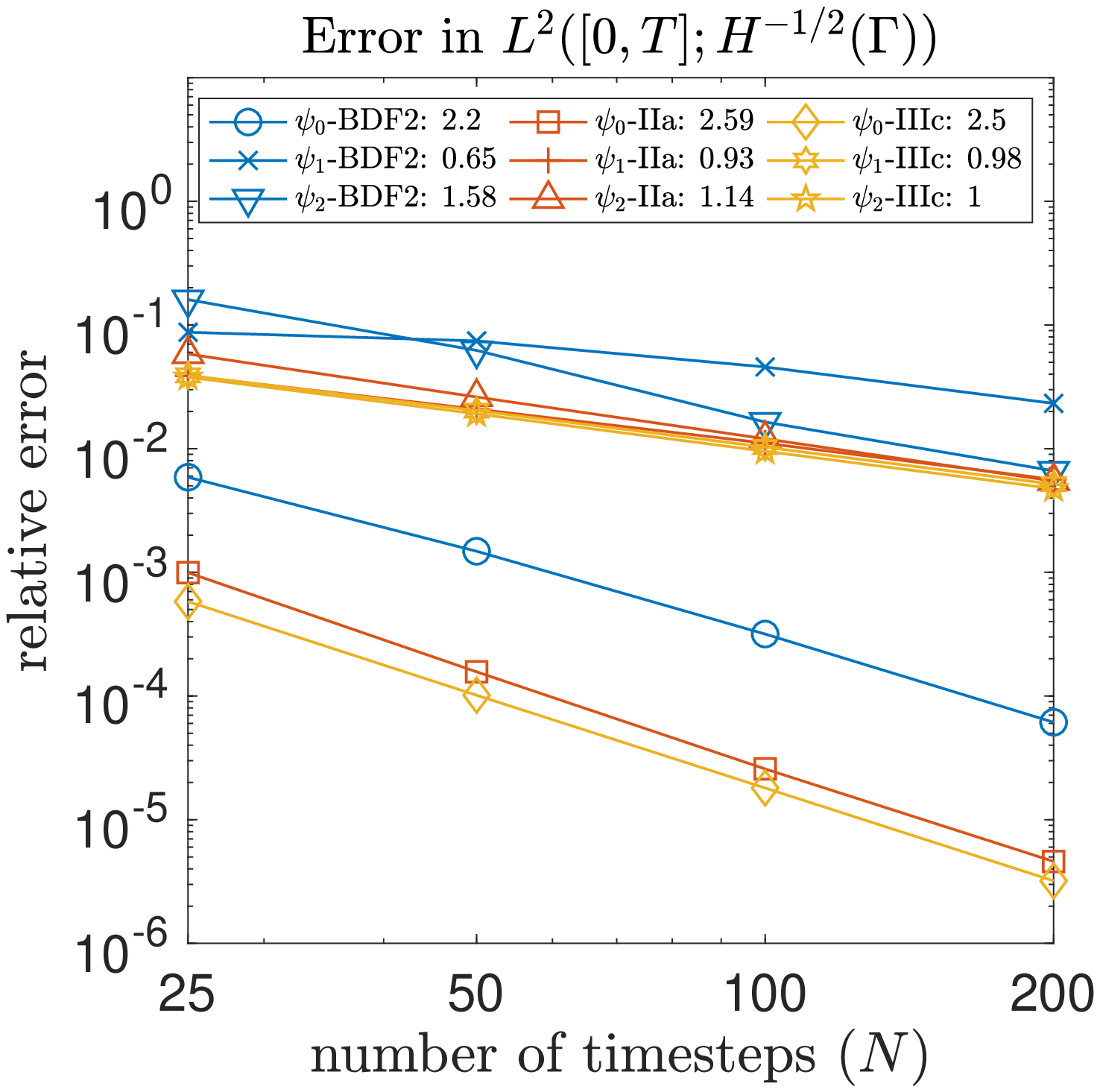}
		\caption{Neumann trace errors} \label{fig:error_neumann_1}
	\end{subfigure}
	\hspace{-0.3cm}
	\begin{subfigure}[c]{0.35\textwidth}
		\centering 
		\includegraphics[width=1.0\linewidth]{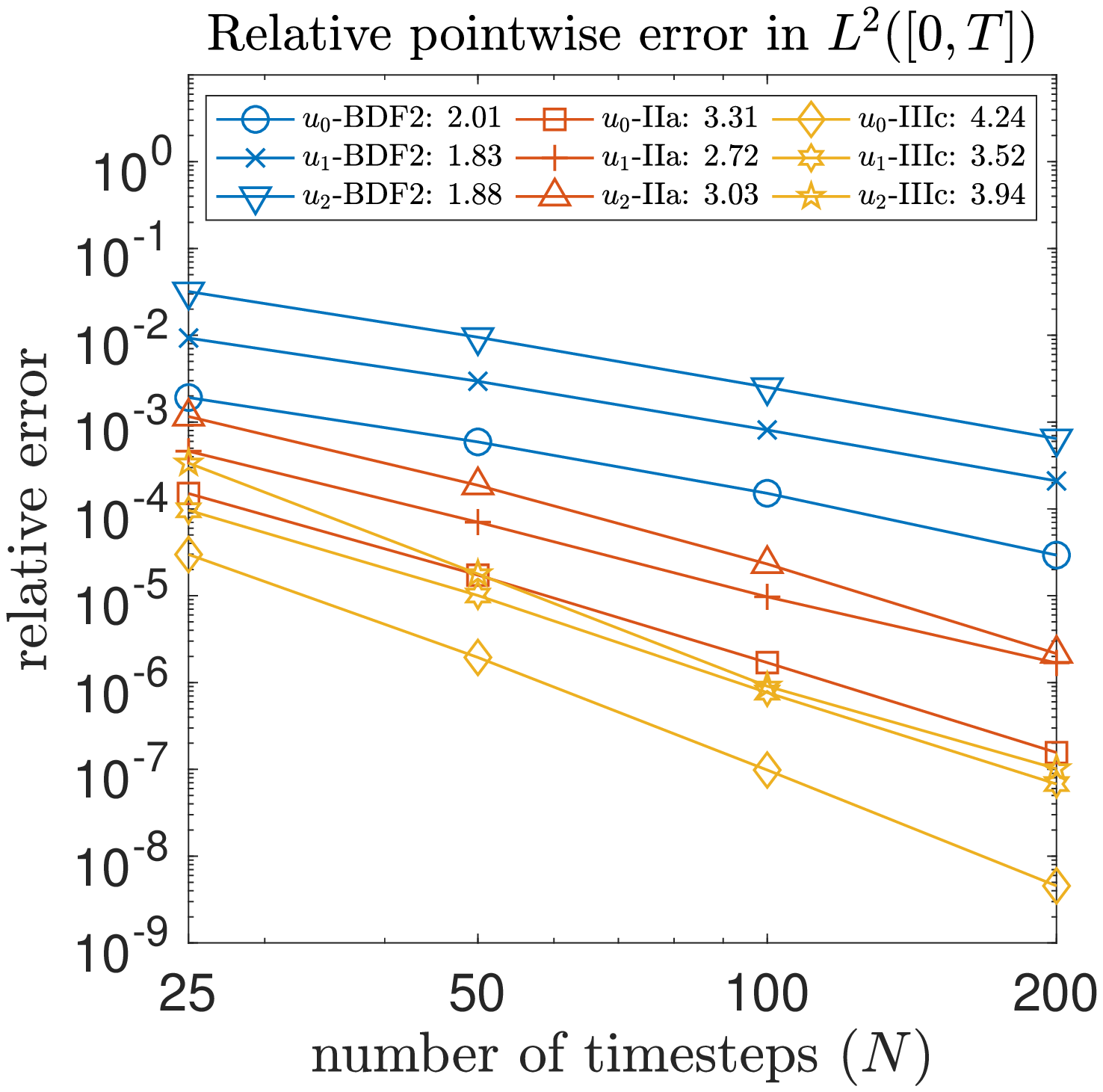}	
		\caption{Field errors} \label{fig:error_U_1}
	\end{subfigure}
	\caption{Relative errors for the two artificial-subdomains case measured with respect to the analytical solution. Errors are computed as  explained in \eqref{eq:trace_error} and \eqref{eq:errorU}. Legends show estimated order of convergence obtained by means of least squares fittings.}
	\label{fig:analytic_sol}
\end{figure}

\subsection{Incident plane wave over a circle with two subdomains\label{sec:example1}}
We now consider as incident field a plane wave coming from $ \Omega_0 $, defined as
\begin{equation}\label{eq:planewave}
u^{\text{inc}}(\bm{x}, t) = f(c_0(t - t_{\text{lag}}) - \bm{x} \cdot \mathbf{d}), \quad f(t) = \sin(\omega t)\eta(t, 0.2, 2).
\end{equation}
The domain is a circle of radius $ r = 0.5 $ divided into two subdomains from Figure \ref{fig:domain_mtf1}, with parameters used are shown in Table \ref{tab:params_test2}.
\begin{table}[t]
	\centering 
	\begin{tabular}{|c|c|c|c|c|c|c|c|}
		\hline\hline 
		$ c_0 $ & $ c_1 $ & $ c_2 $ &$  \omega $ & $ T $ &$ t_{\text{lag}} $ &$ \bm{d} $&$ N_{\text{cheb}} $ \\ \hline 
		$ 1$ & $ 0.5 $ &$ 0.25 $& $ 8 $& $ 10 $ & $ 0.5 $&$ (\sqrt{0.5}, -\sqrt{0.5}) $& $ 80 $\\ \hline 
	\end{tabular}
	\caption{Parameters used in Section \ref{sec:example1}.}
	\label{tab:params_test2}
\end{table}

Error convergence results with respect to a highly resolved solution for each subdomain are shown in Figure \ref{fig:num_sol1}. We observe that the second order convergence for the BDF2 method is not achieved until a high number of timesteps is reached due to the highly oscillatory behavior of the incident field. Runge-Kutta methods show the expected order of convergence for Dirichlet traces and for the scattered field, having even better results for the case of RadauIIa. Neumann traces present a slower convergence, as in previous examples. This behaviour is related to the stage order of convergence in Runge-Kutta methods, which is the same for two-stages RadauIIa and three-stages LobattoIIIc ($ q = 2 $). Some snapshots of the solution are shown in Figure \ref{fig:plot_circle}.

\begin{figure}[t]
	\hspace{-0.7cm}
	\begin{subfigure}[c]{0.35\textwidth}
		\centering 
		\includegraphics[width=1.0\linewidth]{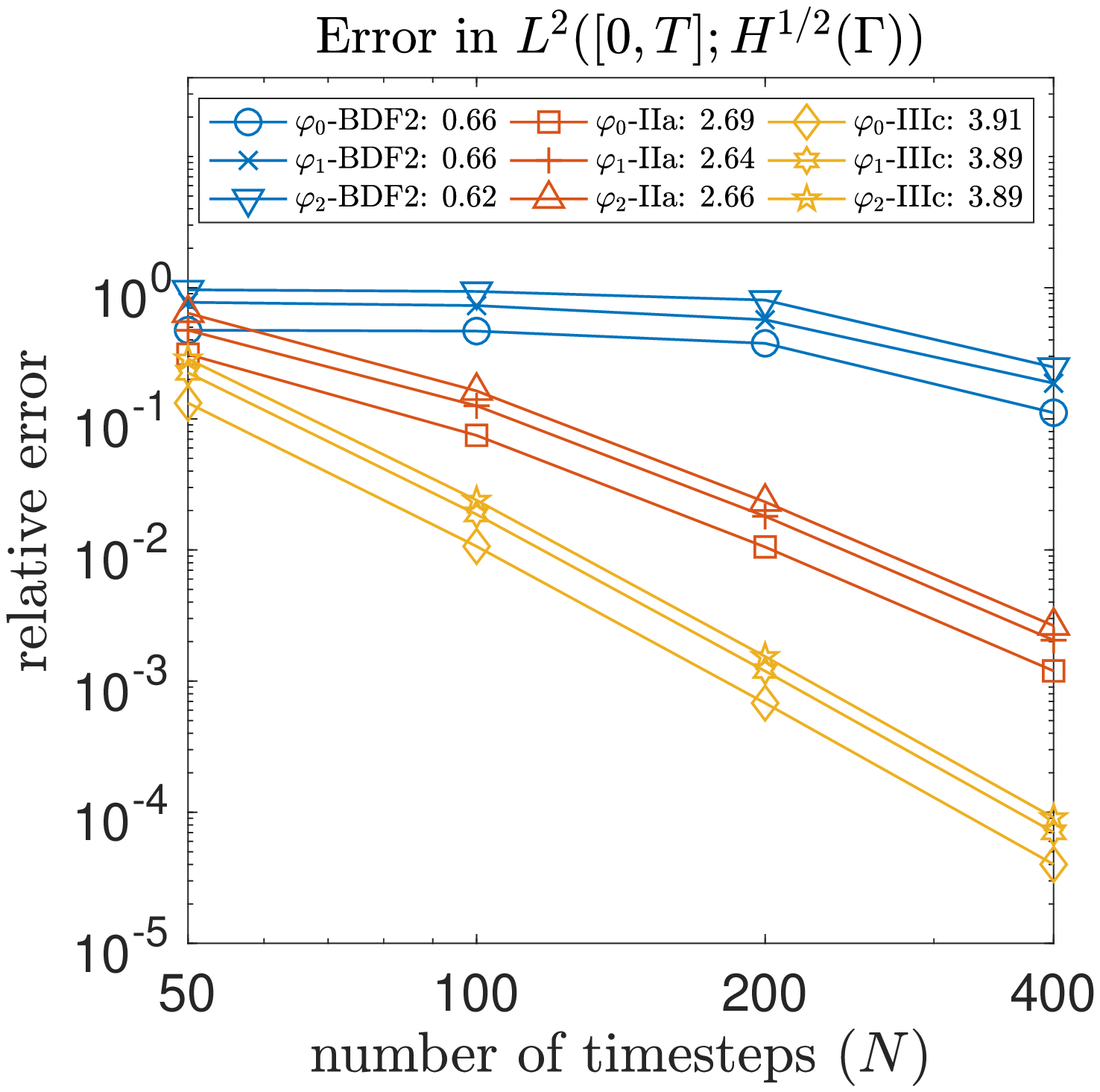}	
		\caption{Dirichlet trace errors} \label{fig:error_dirichlet_2}
	\end{subfigure}
	\hspace{-0.3cm}
	\begin{subfigure}[c]{0.35\textwidth}
		\centering 
		\includegraphics[width=1.0\linewidth]{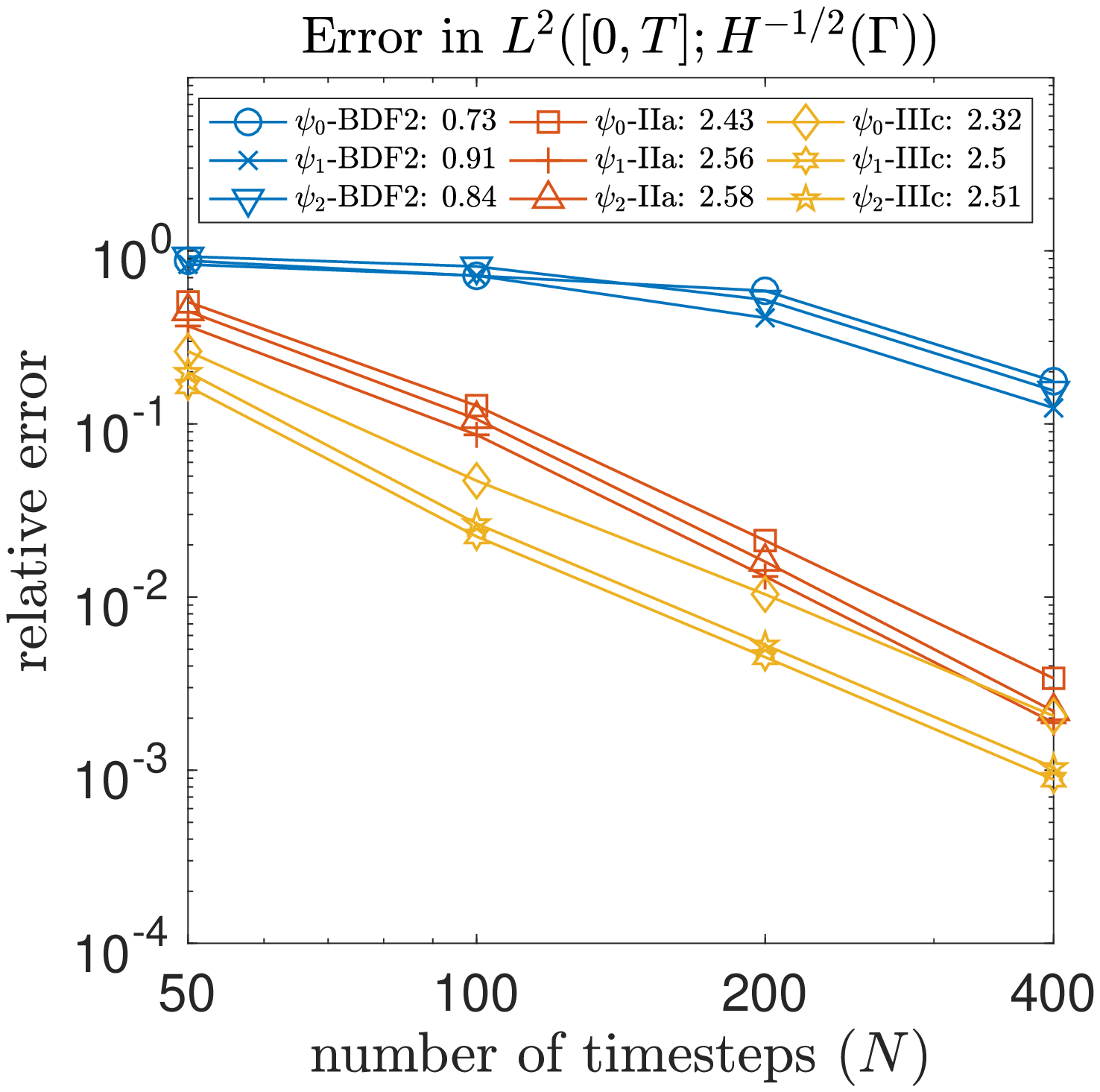}
		\caption{Neumann trace errors} \label{fig:error_neumann_2}
	\end{subfigure}
	\hspace{-0.3cm}
	\begin{subfigure}[c]{0.35\textwidth}
		\centering 
		\includegraphics[width=1.0\linewidth]{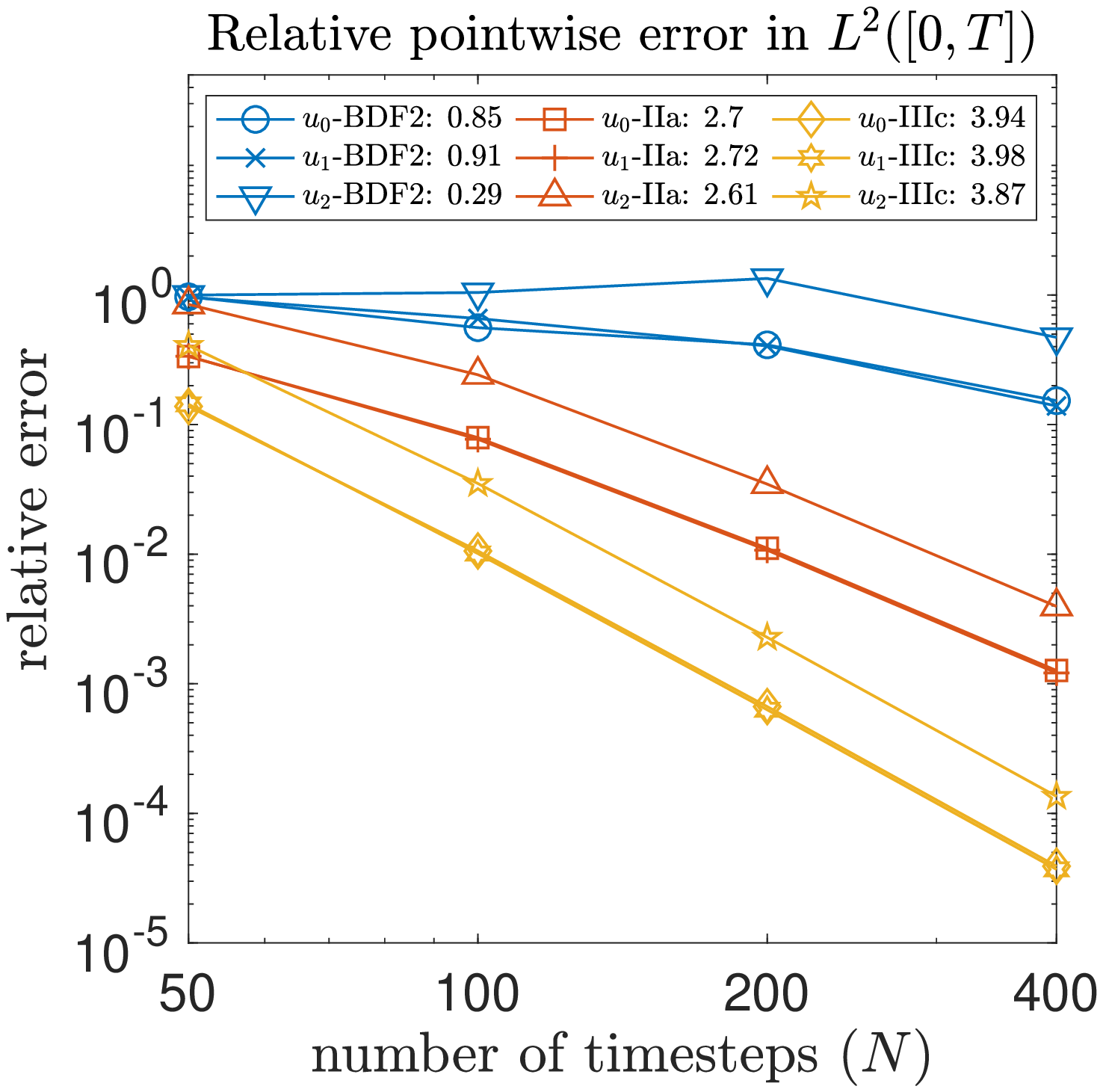}	
		\caption{Field errors} \label{fig:error_U_2}
	\end{subfigure}
	\caption{Relative errors for the two subdomains case measured with respect to a highly resolved solution. Errors are computed as  explained in \eqref{eq:trace_error} and \eqref{eq:errorU}. Legends show methods employed with estimated orders of convergence obtained by means of least squares fittings.}
	\label{fig:num_sol1}
\end{figure}

\begin{figure}[t]
	\centering 
	\begin{tabular}{ccc}
		\includegraphics[width=0.15\linewidth]{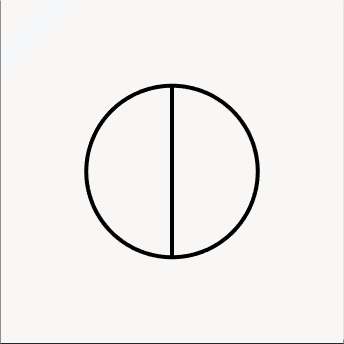}	& \includegraphics[width=0.15\linewidth]{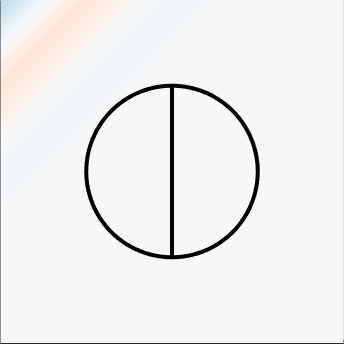} & \includegraphics[width=0.15\linewidth]{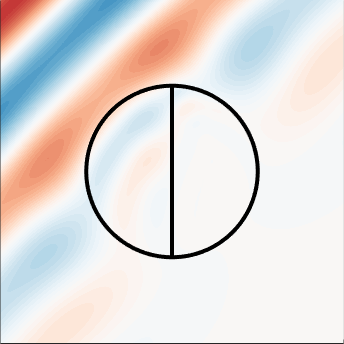} \\ 
		\includegraphics[width=0.15\linewidth]{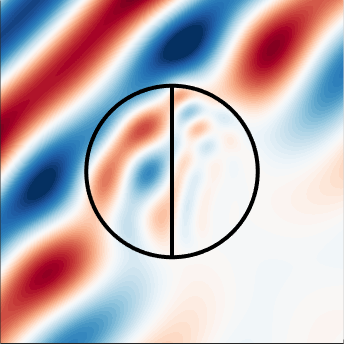}	& \includegraphics[width=0.15\linewidth]{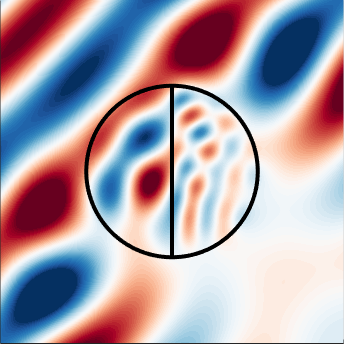} & \includegraphics[width=0.15\linewidth]{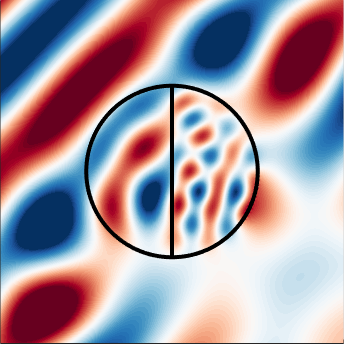} \\ 
		\includegraphics[width=0.15\linewidth]{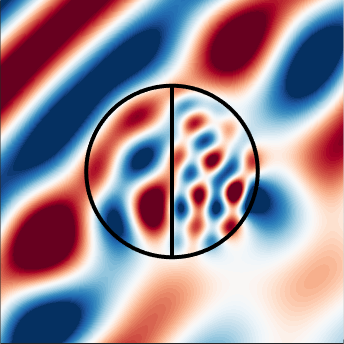}	& \includegraphics[width=0.15\linewidth]{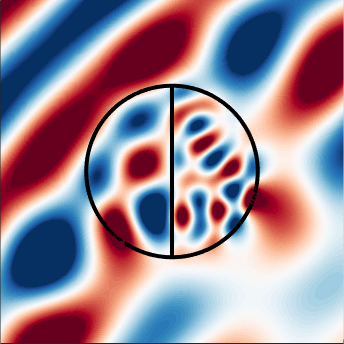} &  \includegraphics[width=0.15\linewidth]{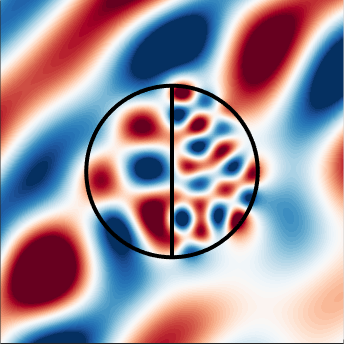}\\ 
	\end{tabular} 
	\caption{Snapshots of the computed field from problem in Section \ref{sec:example1} for the times $$ t = 0,\ 1.25, \ 2.5, \ 3.75, \ 5, \ 6.25, \ 7.5, \ 8.75, \ 10. $$}
	\label{fig:plot_circle}
\end{figure}

\subsection{Incident Plane Wave over a Square with Four Subdomains}
\label{sec:example2}
In this experiment, the incident field is the same plane wave from \eqref{eq:planewave} coming from $ \Omega_0 $ with parameters used are $ t_{lag}= 0.5, \ \bm{d}= (\sqrt{0.5}, -\sqrt{0.5})$ and $ \omega = 8. $ The domain is a square of side length $ a = 1 $ divided into four subdomains (see Figure \ref{fig:domain_mtf2} ). Wavespeeds on each subdomain are $ c_0 = 1, c_1 = 0.5, c_2 = 0.25, c_3 = 0.5, c_4 = 0.25$. 

\begin{table}[t]
	\centering 
	\caption{Parameters used  in Section \ref{sec:example2}.}
	\begin{tabular}{|c|c|c|c|c|c|c|c|c|c|}
		\hline\hline 
		$ c_0 $ & $ c_1 $ & $ c_2 $ &$ c_3 $&$ c_4 $&$  \omega $ & $ T $ &$ t_{\text{lag}} $ &$ \bm{d} $&$ N_{\text{cheb}} $ \\ \hline 
		$ 1$ & $ 0.5 $ &$ 0.25 $& $ 0.5 $&$ 0.25 $ & $ 8 $& $ 10 $ & $ 0.5 $&$ (\sqrt{0.5}, -\sqrt{0.5}) $& $ 20 $\\ \hline 
	\end{tabular}
	\label{tab:params_test3}
\end{table}

Convergence results for each interface are shown in Figure \ref{fig:num_sol2}. We observe a similar behaviour to the previous example, with slow convergence for the BDF2 method. Snapshots of the volume solution are displayed in Figure \ref{fig:plot_square}.

\begin{figure}[t]
	\hspace{-0.7cm}
	\begin{subfigure}[c]{0.35\textwidth}
		\centering 
		\includegraphics[width=1.0\linewidth]{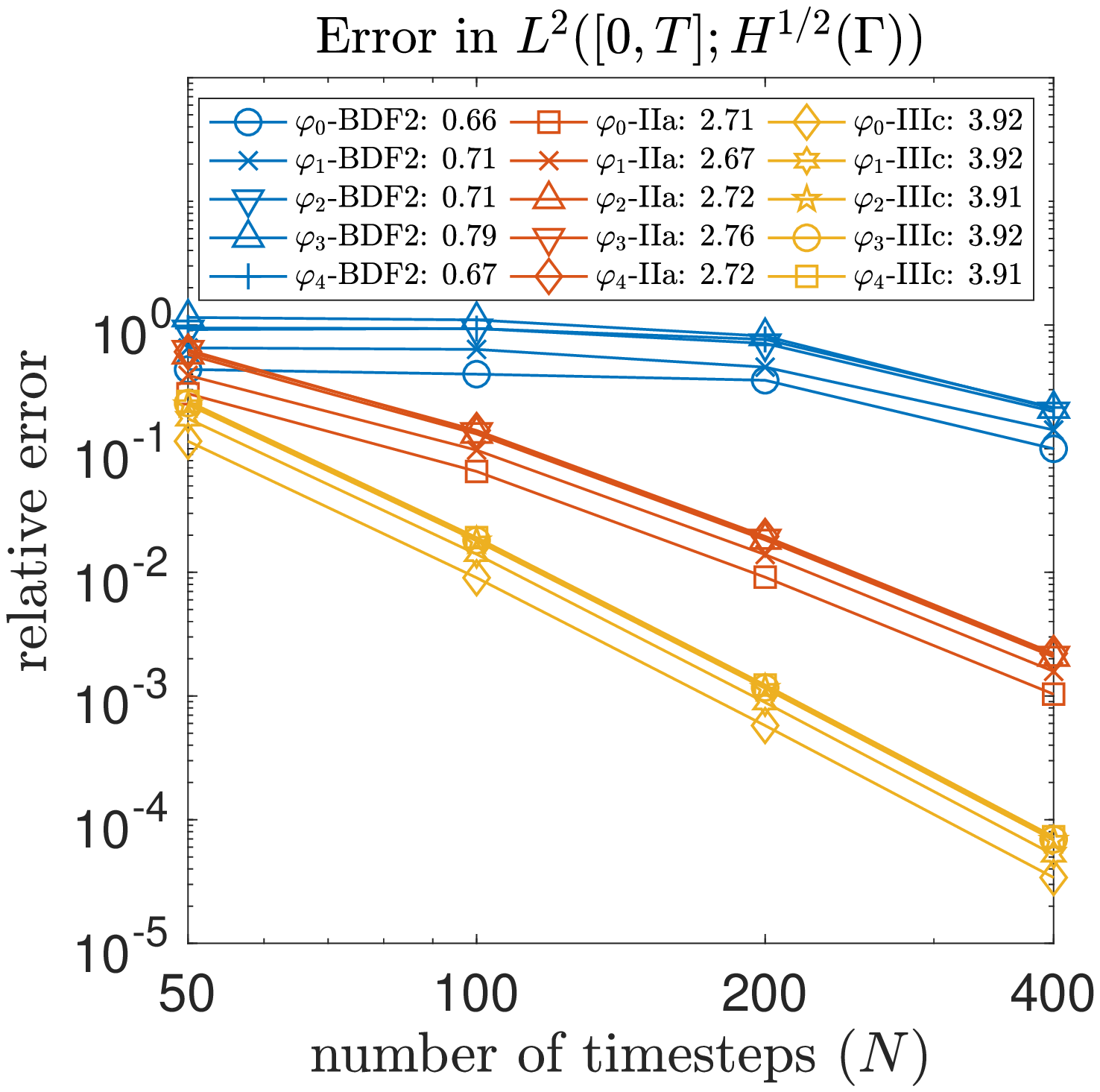}	
		\caption{Dirichlet trace errors} \label{fig:error_dirichlet_3}
	\end{subfigure}
	\hspace{-0.3cm}
	\begin{subfigure}[c]{0.35\textwidth}
		\centering 
		\includegraphics[width=1.0\linewidth]{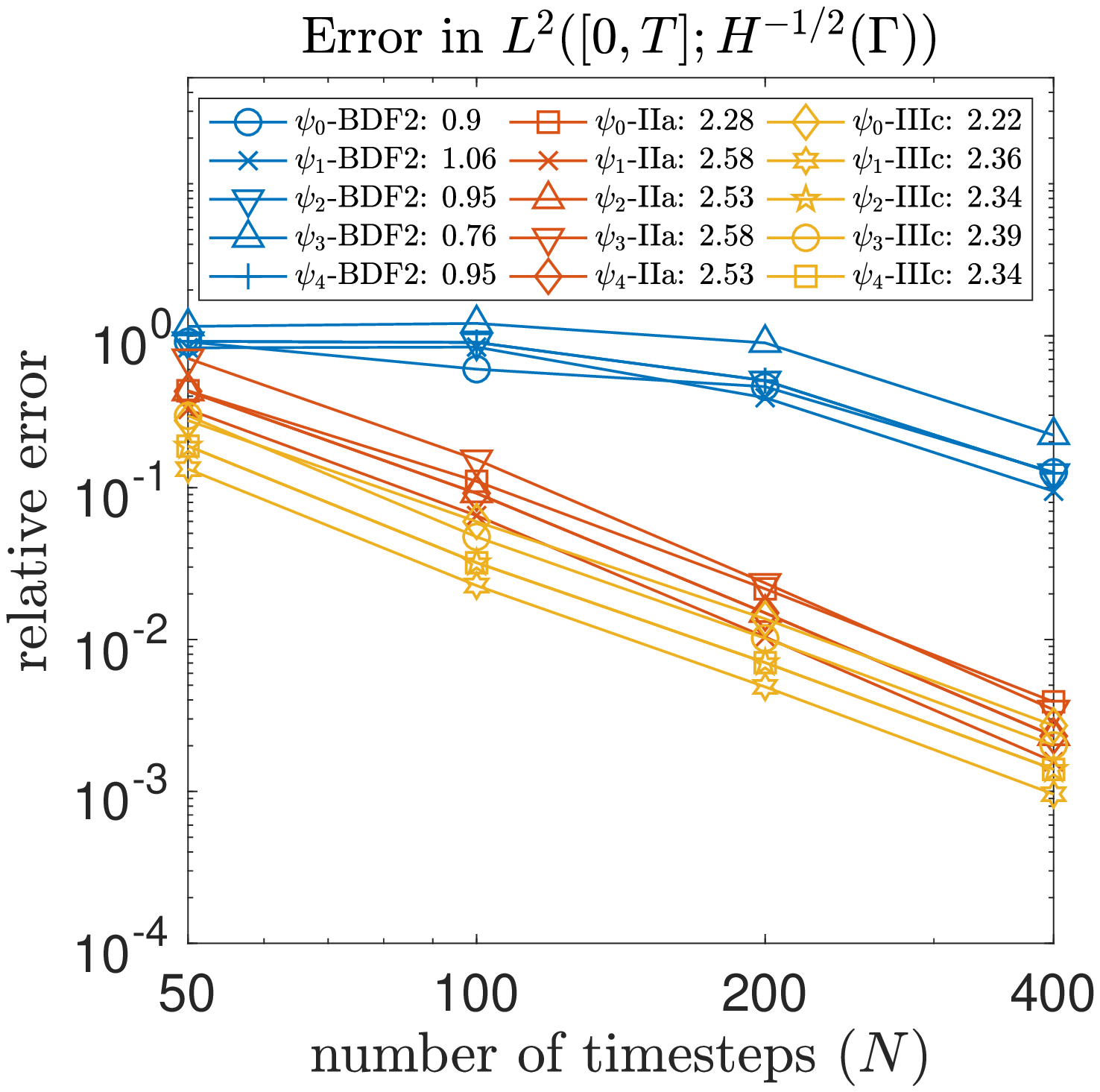}
		\caption{Neumann trace errors} \label{fig:error_neumann_3}
	\end{subfigure}
	\hspace{-0.3cm}
	\begin{subfigure}[c]{0.35\textwidth}
		\centering 
		\includegraphics[width=1.0\linewidth]{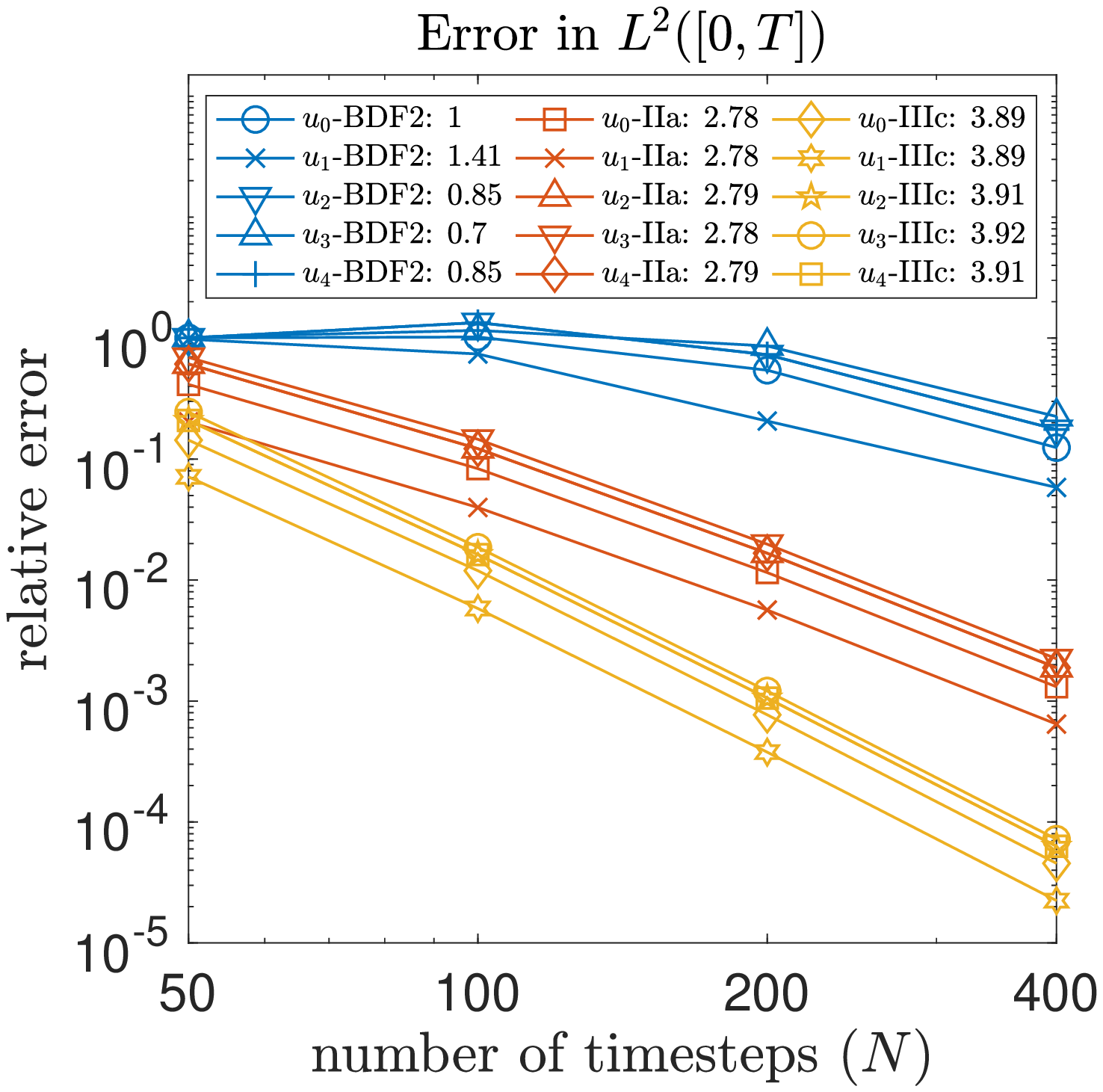}	
		\caption{Field errors} \label{fig:error_U_3}
	\end{subfigure}
	\caption{Relative errors for the four subdomains case measured with respect to a highly resolved solution. Errors are computed as  explained in \eqref{eq:trace_error} and \eqref{eq:errorU}. Legends indicate estimated order of convergence obtained by means of least squares fittings.}
	\label{fig:num_sol2}
\end{figure}

\begin{figure}[t]
	\centering
	\begin{tabular}{ccc}
		\includegraphics[width=0.15\linewidth]{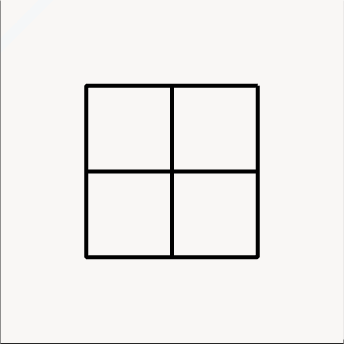}	& \includegraphics[width=0.15\linewidth]{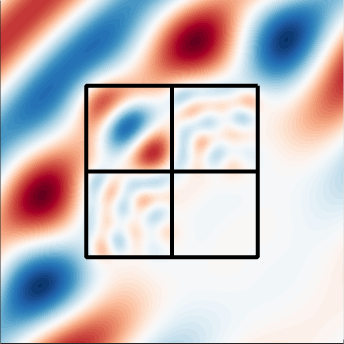} & \includegraphics[width=0.15\linewidth]{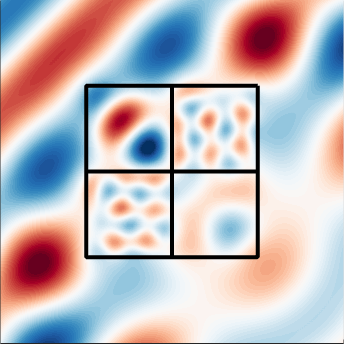} \\ 
		\includegraphics[width=0.15\linewidth]{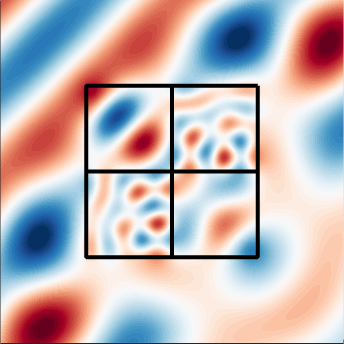}	& \includegraphics[width=0.15\linewidth]{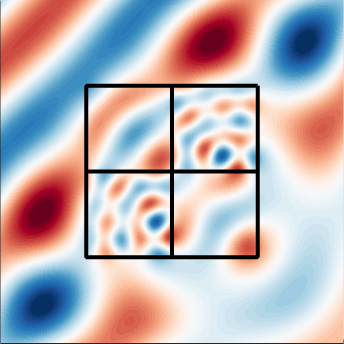} & \includegraphics[width=0.15\linewidth]{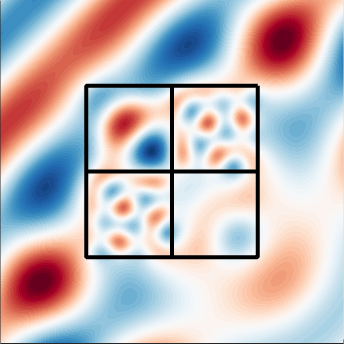} \\ 
		\includegraphics[width=0.15\linewidth]{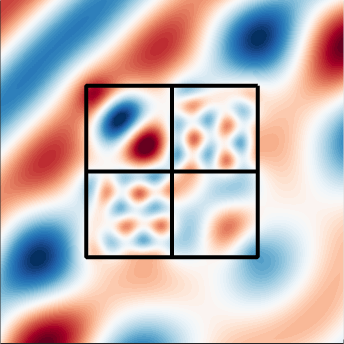}	& \includegraphics[width=0.15\linewidth]{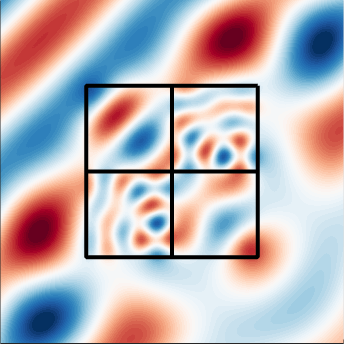} &  \includegraphics[width=0.15\linewidth]{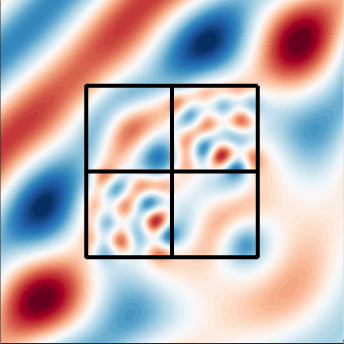}\\ 
	\end{tabular} 
	\caption{Snapshots of the computed field from problem in Section \ref{sec:example2} for times $$ t = 0,\ 1.25, \ 2.5, \ 3.75, \ 5, \ 6.25, \ 7.5, \ 8.75, \ 10.0. $$}
	\label{fig:plot_square}
\end{figure}

\subsection{Incident plane wave over kite with two subdomains}
\label{sec:example3}
For our last experiment, the incident field is the same plane wave from \eqref{eq:planewave} coming from $ \Omega_0 $. The domain is a smooth kite parametrized as follows
\begin{equation}
\begin{pmatrix}
x(t)\\ y(t)
\end{pmatrix} = \begin{pmatrix}
\cos(t) + 0.65\cos(2t)\\ 
\sin(t)
\end{pmatrix}, \quad t\in [0, 2\pi]
\end{equation}
divided into two subdomains $ \Omega_1 $ and $ \Omega_2 $ (corresponding to upper and lower halves) with symmetry over the $ x- $axis. Parameters used are shown in Table \ref{tab:params_test4}. Convergence results for each interface are shown in Figure \ref{fig:num_sol3}.  Snapshots of the volume solution are displayed in Figure \ref{fig:plot_kite}.

\begin{table}[t]
	\centering 
	\caption{Parameters used in Section \ref{sec:example3}.}
	\begin{tabular}{|c|c|c|c|c|c|c|c|}
		\hline\hline 
		$ c_0 $ & $ c_1 $ & $ c_2 $ &$  \omega $ & $ T $ &$ t_{\text{lag}} $ &$ \bm{d} $&$ N_{\text{cheb}} $ \\ \hline 
		$ 1$ & $ 0.5 $ &$ 0.25 $& $ 8 $& $ 10 $ & $ 0.5 $&$ (1,0) $& $ 80 $\\ \hline 
	\end{tabular}
	\label{tab:params_test4}
\end{table}

\begin{figure}[t]
	\hspace{-0.7cm}
	\begin{subfigure}[c]{0.35\textwidth}
		\centering 
		\includegraphics[width=1.0\linewidth]{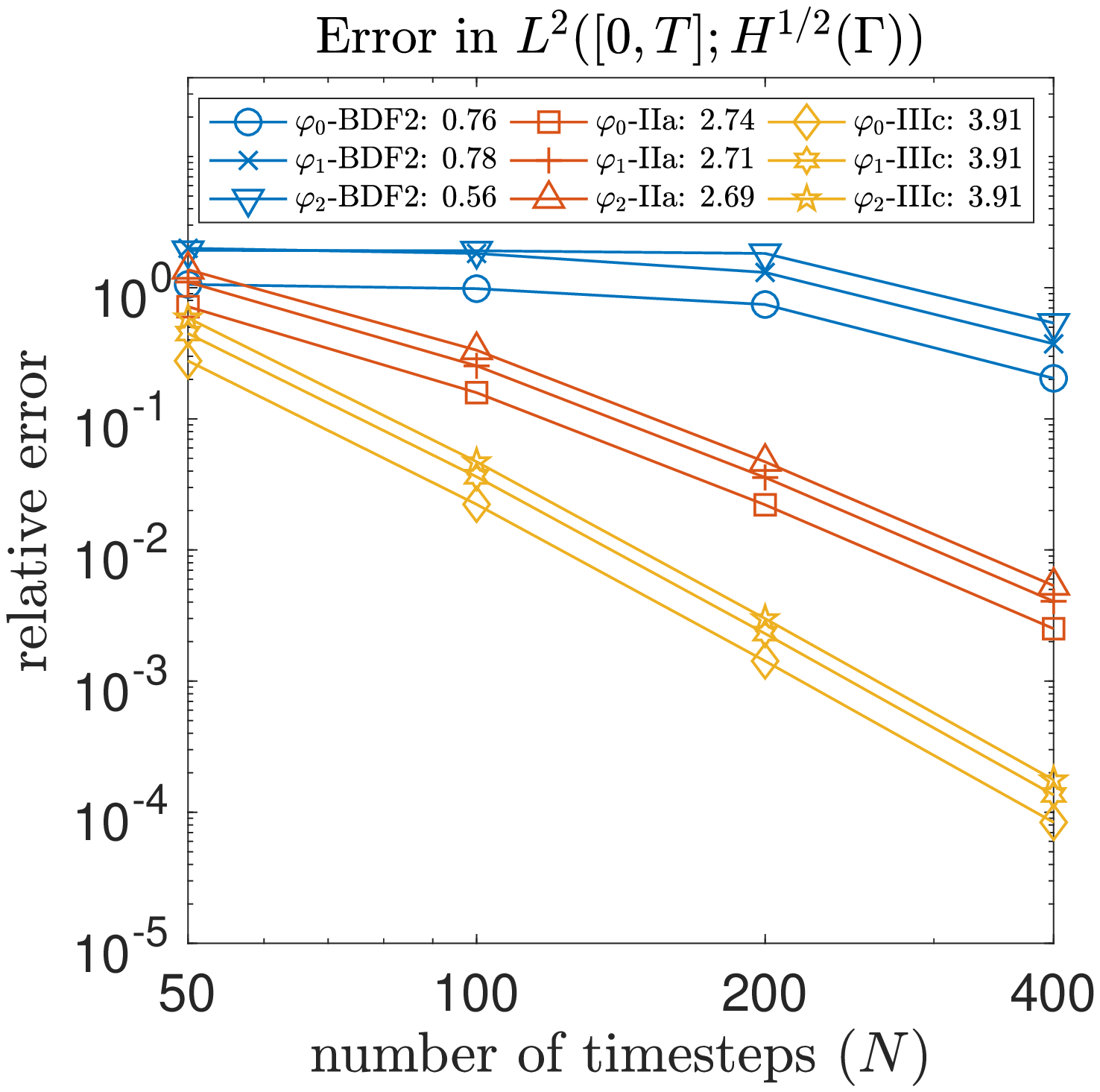}	
		\caption{Dirichlet trace errors} \label{fig:error_dirichlet_kite}
	\end{subfigure}
	\hspace{-0.3cm}
	\begin{subfigure}[c]{0.35\textwidth}
		\centering 
		\includegraphics[width=1.0\linewidth]{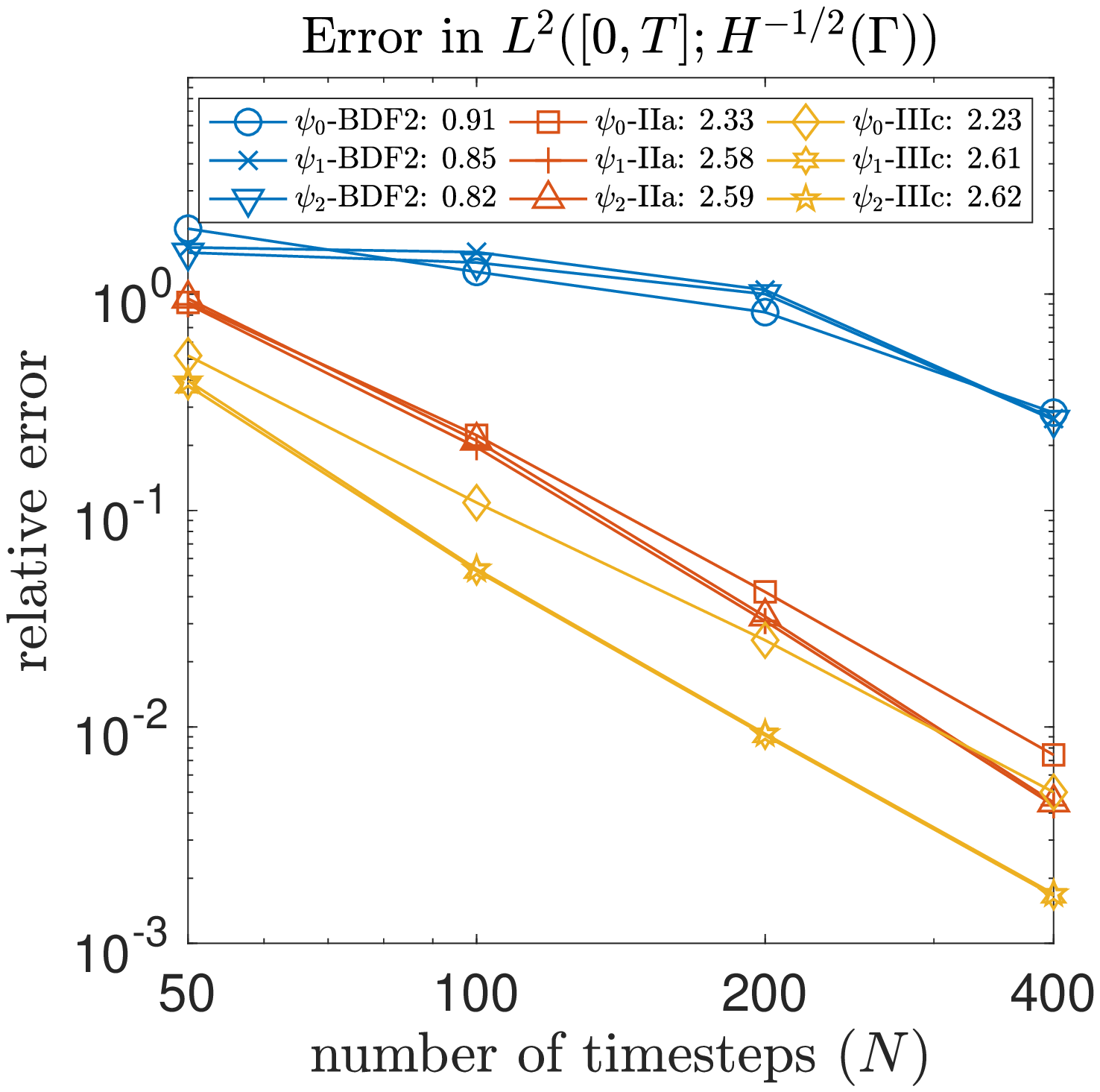}
		\caption{Neumann trace errors} \label{fig:error_neumann_kite}
	\end{subfigure}
	\hspace{-0.3cm}
	\begin{subfigure}[c]{0.35\textwidth}
		\centering 
		\includegraphics[width=1.0\linewidth]{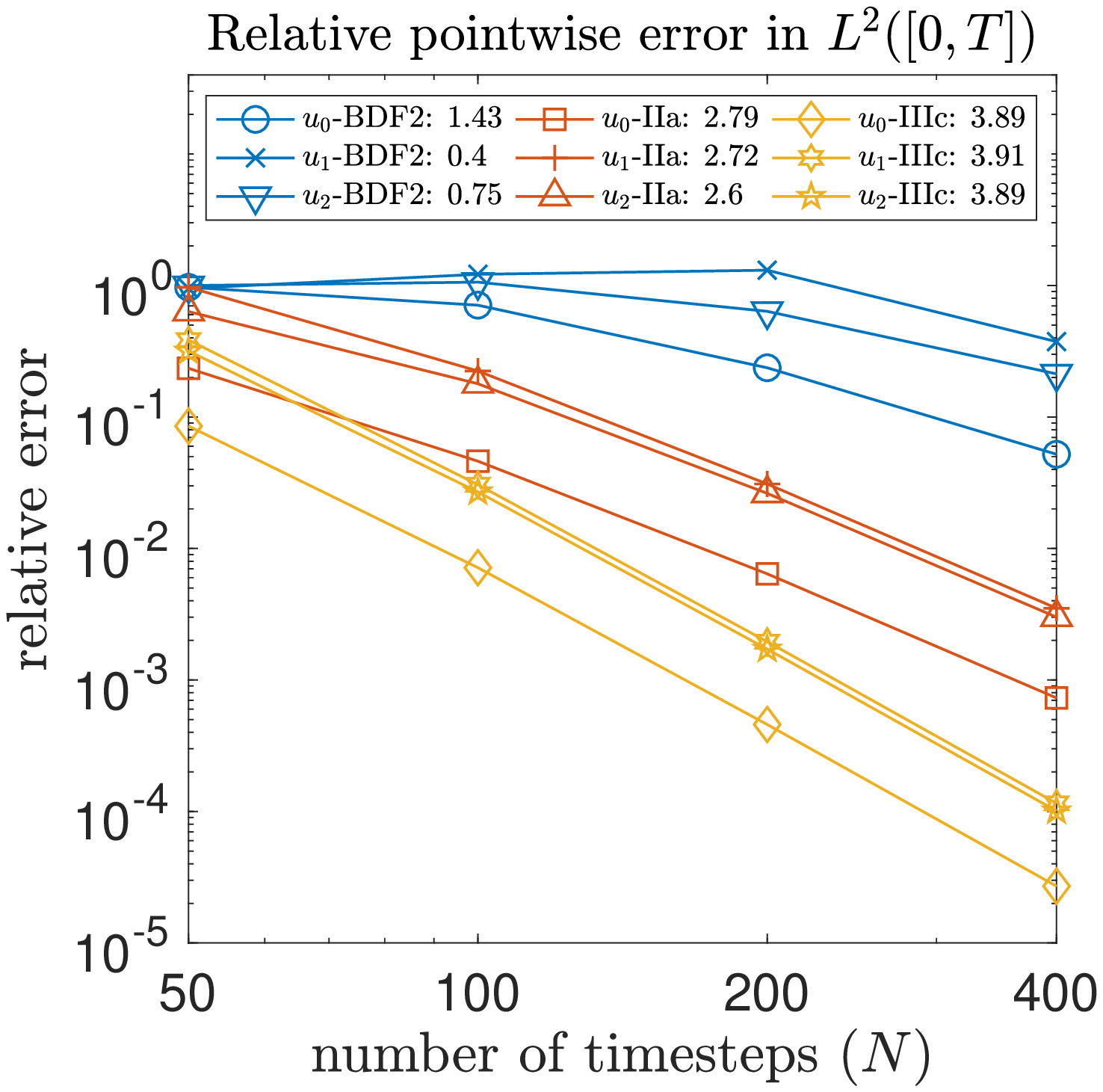}	
		\caption{Field errors} \label{fig:error_U_kite}
	\end{subfigure}
	\caption{Relative errors for the two subdomains case in a kite-shaped domain, measured with respect to a highly resolved solution. Errors are computed as  explained in \eqref{eq:trace_error} and \eqref{eq:errorU}. Legends indicate estimated order of convergence obtained by means of least squares fittings.}
	\label{fig:num_sol3}
\end{figure}

\begin{figure}[t]
	\centering
	\begin{tabular}{ccc}
		\includegraphics[width=0.15\linewidth]{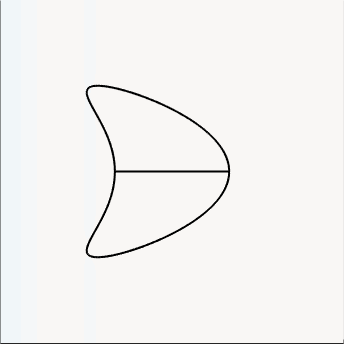}	& \includegraphics[width=0.15\linewidth]{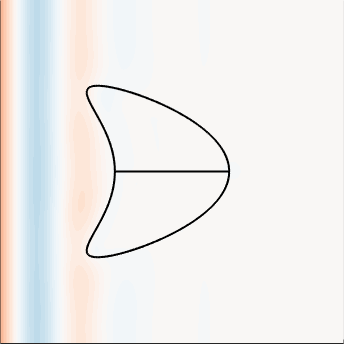} & \includegraphics[width=0.15\linewidth]{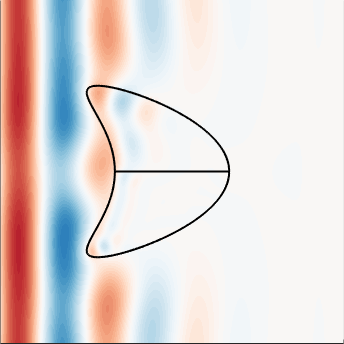} \\ 
		\includegraphics[width=0.15\linewidth]{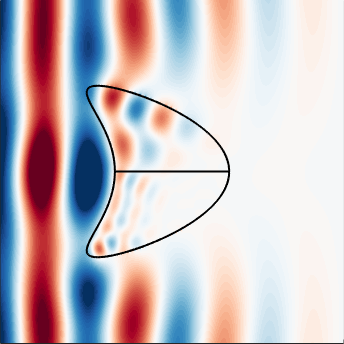}	& \includegraphics[width=0.15\linewidth]{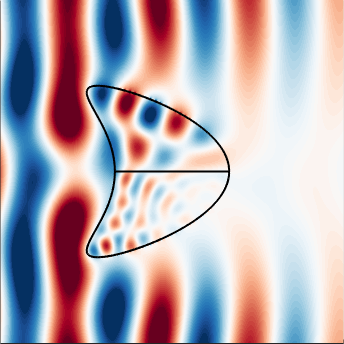} & \includegraphics[width=0.15\linewidth]{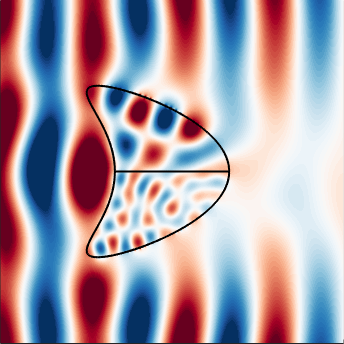} \\ 
		\includegraphics[width=0.15\linewidth]{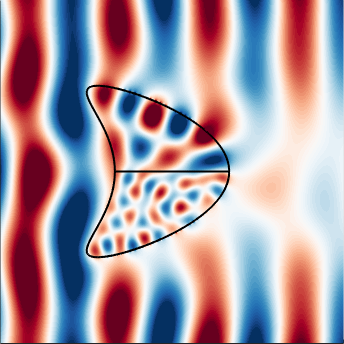}	& \includegraphics[width=0.15\linewidth]{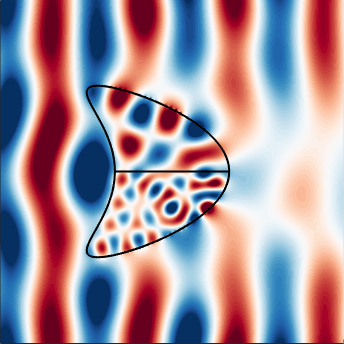} &  \includegraphics[width=0.15\linewidth]{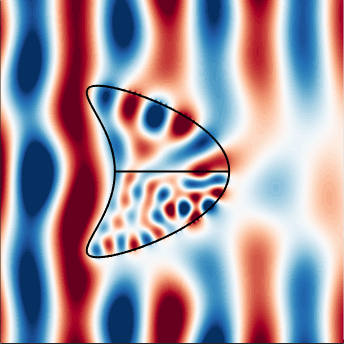}\\ 
	\end{tabular} 
	\caption{Snapshots of the computed field from problem in Section \ref{sec:example3} for times 
		$$ t = 0.5,\ 1.5, \ 2.5, \ 3.5, \ 4.5, \ 5.5, \ 6.5, \ 7.5, \ 8.5. $$}
	\label{fig:plot_kite}
\end{figure}

\section{Concluding remarks}
We solve acoustic wave transmission problems in 2D over composite scatterers by means of CQ methods coupled to a non-conforming spectral Galerkin discretization in space. We showed that Runge-Kutta CQ constitutes an efficient and accurate method, preferable to multistep methods due to their higher order convergence rates. Although reduced order can be expected, this was only observed for the convergence of Neumann traces. For every example, convergence of two-stage Radau and three-stage Lobatto coincides for Neumann traces, which can be explained with the fact that both methods have exactly the same stage order $ q = 2 $. 

The use of high-order spatial discretizations for frequency domain problems is mandatory in order to achieve accurate results and take advantage of the capabilities of Runge-Kutta methods. This is achieved in space via a non-conforming spectral Galerkin discretization. Chebyshev polynomials are also suitable functions for the framework of piecewise/broken Sobolev spaces defined over the boundary of each subdomain.

Future work will be focused on developing a convergence theory for this formulation. Although very important advances have been made in previous works in \cite{hassell2017new,qiu2016time,qiu2016costabel,rieder2017convolution}, the current theoretical framework is well suited only for Galerkin-BEM on standard Sobolev spaces over the boundary of Lipschitz domains. Unfortunately, it is not clear how to extend this to piecewise or broken spaces considered for trial and test functions in the MTF, or even the Petrov-Galerkin formulation. 

\section*{Acknowledgments}

This research was funded by FONDECYT Regular 1171491.

\bibliographystyle{abbrv}      
\bibliography{References2}   

\begin{thebibliography}{10}

\bibitem{banjai2011error}
L.~Banjai and C.~Lubich.
\newblock An error analysis of {R}unge--{K}utta {C}onvolution {Q}uadrature.
\newblock {\em BIT Numerical Mathematics}, 51(3):483--496, 2011.

\bibitem{banjai2011runge}
L.~Banjai, C.~Lubich, and J.~M. Melenk.
\newblock Runge--{K}utta {C}onvolution {Q}uadrature for operators arising in
  {W}ave propagation.
\newblock {\em Numerische Mathematik}, 119(1):1--20, 2011.

\bibitem{banjai2012runge}
L.~Banjai, M.~Messner, and M.~Schanz.
\newblock Runge--{K}utta {C}onvolution {Q}uadrature for the {B}oundary
  {E}lement {M}ethod.
\newblock {\em Computer methods in applied mechanics and engineering},
  245:90--101, 2012.

\bibitem{banjai2008rapid}
L.~Banjai and S.~Sauter.
\newblock Rapid {S}olution of the {W}ave {E}quation in {U}nbounded {D}omains.
\newblock {\em SIAM Journal on Numerical Analysis}, 47(1):227--249, 2008.

\bibitem{banjai2012wave}
L.~Banjai and M.~Schanz.
\newblock Wave propagation problems treated with {C}onvolution {Q}uadrature and
  {B}{E}{M}.
\newblock In {\em {F}ast {B}oundary {E}lement {M}ethods in {E}ngineering and
  {I}ndustrial {A}pplications}, pages 145--184. Springer, 2012.

\bibitem{CHJ13}
X.~Claeys, R.~Hiptmair, and C.~Jerez-Hanckes.
\newblock Multitrace {B}oundary {I}ntegral {E}quations.
\newblock In {\em Direct and inverse problems in wave propagation and
  applications}, volume~14 of {\em Radon Ser. Comput. Appl. Math.}, pages
  51--100. De Gruyter, Berlin, 2013.

\bibitem{CHJ15}
X.~Claeys, R.~Hiptmair, C.~Jerez-Hanckes, and S.~Pintarelli.
\newblock Novel {M}ulti-{T}race {B}oundary {I}ntegral {E}quations for
  {T}ransmission {B}oundary {V}alue {P}roblems.
\newblock In A.~S. Fokas and B.~Pelloni, editors, {\em Unified Transform for
  Boundary Value Problems: Applications and Advances}, pages 227--258.
  Philadelphia, SIAM, 2015.

\bibitem{spindler2015second}
X.~Claeys, R.~Hiptmair, and E.~Spindler.
\newblock A second-kind {G}alerkin {B}oundary {E}lement method for scattering
  at composite objects.
\newblock {\em BIT Numerical Mathematics}, 55(1):33--57, 2015.

\bibitem{davies1998stability}
P.~J. Davies.
\newblock A stability analysis of a time marching scheme for the general
  surface electric field integral equation.
\newblock {\em Applied Numerical Mathematics}, 27(1):33--57, 1998.

\bibitem{davies1997averaging}
P.~J. Davies and D.~B. Duncan.
\newblock Averaging techniques for time-marching schemes for retarded potential
  integral equations.
\newblock {\em Applied Numerical Mathematics}, 23(3):291--310, 1997.

\bibitem{davies2004stability}
P.~J. Davies and D.~B. Duncan.
\newblock Stability and convergence of collocation schemes for retarded
  potential integral equations.
\newblock {\em SIAM Journal on Numerical Analysis}, 42(3):1167--1188, 2004.

\bibitem{davies2005stability}
P.~J. Davies, D.~B. Duncan, and B.~Zubik-Kowal.
\newblock The stability of numerical approximations of the time domain current
  induced on thin wire and strip antennas.
\newblock {\em Applied numerical mathematics}, 55(1):48--68, 2005.

\bibitem{EFH19}
S.~Eberle, F.~Florian, R.~Hiptmair, and S.~Sauter.
\newblock A stable boundary integral formulation of an acoustic wave
  transmission problem with mixed boundary conditions.
\newblock {\em SIAM Journal on Mathematical Analysis}, 53(2):1492--1508, 2021.

\bibitem{grisvard2011elliptic}
P.~Grisvard.
\newblock {\em Elliptic problems in nonsmooth domains}.
\newblock SIAM, 2011.

\bibitem{hassell2016convolution}
M.~Hassell and F.-J. Sayas.
\newblock Convolution {Q}uadrature for {W}ave simulations.
\newblock In {\em Numerical simulation in physics and engineering}, pages
  71--159. Springer, 2016.

\bibitem{hassell2017new}
M.~E. Hassell, T.~Qiu, T.~S{\'a}nchez-Vizuet, F.-J. Sayas, et~al.
\newblock A new and improved analysis of the time domain boundary integral
  operators for the acoustic wave equation.
\newblock {\em Journal of Integral Equations and Applications}, 29(1):107--136,
  2017.

\bibitem{HJH18}
F.~Henr\'iquez and C.~Jerez-Hanckes.
\newblock Multiple {T}races {F}ormulation and {S}emi-{I}mplicit {S}cheme for
  {M}odelling {B}iological {C}ells under {E}lectrical {S}timulation.
\newblock {\em ESAIM Mathematical Modelling and Numerical Analysis},
  52(2):659--703, 2018.

\bibitem{HJA16}
F.~Henr{\'\i}quez, C.~Jerez-Hanckes, and F.~Altermatt.
\newblock Boundary {I}ntegral {F}ormulation and {S}emi-{I}mplicit {S}cheme
  coupling for {M}odeling {C}ells under {E}lectrical {S}timulation.
\newblock {\em Numerische Mathematik}, 136:101--145, 2017.

\bibitem{hiptmair2012multiple}
R.~Hiptmair and C.~Jerez-Hanckes.
\newblock Multiple {T}races {B}oundary {I}ntegral {F}ormulation for {H}elmholtz
  {T}ransmission {P}roblems.
\newblock {\em Advances in Computational Mathematics}, 37(1):39--91, 2012.

\bibitem{JNU2018}
C.~Jerez-Hanckes, S.~Nicaise, and C.~Urzúa-Torres.
\newblock Fast spectral galerkin method for logarithmic singular equations on a
  segment.
\newblock {\em Journal of Computational Mathematics}, 36(1):128--158, 2018.

\bibitem{jerez2017multitrace}
C.~Jerez-Hanckes, C.~P{\'e}rez-Arancibia, and C.~Turc.
\newblock Multitrace/singletrace formulations and {D}omain {D}ecomposition
  {M}ethods for the solution of {H}elmholtz transmission problems for bounded
  composite scatterers.
\newblock {\em Journal of Computational Physics}, 350:343--360, 2017.

\bibitem{jerez2015local}
C.~Jerez-Hanckes, J.~Pinto, and S.~Tournier.
\newblock Local multiple traces formulation for high-frequency scattering
  problems.
\newblock {\em Journal of Computational and Applied Mathematics}, 289:306--321,
  2015.

\bibitem{jerez2016local}
C.~Jerez-Hanckes, J.~Pinto, and S.~Tournier.
\newblock Local {M}ultiple {T}races {F}ormulation for {H}igh-{F}requency
  {S}cattering {P}roblems by {S}pectral {E}lements.
\newblock In {\em Scientific Computing in Electrical Engineering}, pages
  73--82. Springer, 2016.

\bibitem{lopez2013generalized}
M.~Lopez-Fernandez and S.~Sauter.
\newblock Generalized {C}onvolution {Q}uadrature with variable time stepping.
\newblock {\em IMA Journal of Numerical Analysis}, 33(4):1156--1175, 2013.

\bibitem{lopez2016generalized}
M.~Lopez-Fernandez and S.~Sauter.
\newblock Generalized {C}onvolution {Q}uadrature based on {R}unge-{K}utta
  methods.
\newblock {\em Numerische Mathematik}, 133(4):743--779, 2016.

\bibitem{lubich1988convolution}
C.~Lubich.
\newblock {Convolution {Q}uadrature and {D}iscretized {O}perational {C}alculus.
  I}.
\newblock {\em Numerische Mathematik}, 52(2):129--145, 1988.

\bibitem{lubich1988convolution_II}
C.~Lubich.
\newblock {Convolution {Q}uadrature and {D}iscretized {O}perational {C}alculus.
  II}.
\newblock {\em Numerische Mathematik}, 52(4):413--425, 1988.

\bibitem{lubichRK}
C.~Lubich and A.~Ostermann.
\newblock Runge-{K}utta {M}ethods for {P}arabolic {E}quations and {C}onvolution
  {Q}uadrature.
\newblock {\em Mathematics of Computation}, 60(201):105--131, 1993.

\bibitem{mclean2000strongly}
W.~C.~H. McLean.
\newblock {\em Strongly {E}lliptic systems and {B}oundary {I}ntegral
  {E}quations}.
\newblock Cambridge university press, 2000.

\bibitem{qiu2016time}
T.~Qiu.
\newblock {\em Time {D}omain {B}oundary {I}ntegral {E}quation methods in
  {A}coustics, {H}eat diffusion and {E}lectromagnetism}.
\newblock PhD thesis, University of Delaware, 2016.

\bibitem{qiu2016costabel}
T.~Qiu and F.-J. Sayas.
\newblock The {C}ostabel-{S}tephan system of {B}oundary {I}ntegral {E}quations
  in the {T}ime {D}omain.
\newblock {\em Mathematics of Computation}, 85(301):2341--2364, 2016.

\bibitem{rieder2017convolution}
D.-I.~A. Rieder.
\newblock {\em Convolution {Q}uadrature and {B}oundary {E}lement {M}ethods in
  wave propagation}.
\newblock PhD thesis, Technische Universit{\"a}t Wien, 2017.

\bibitem{sayas2016retarded}
F.-J. Sayas.
\newblock {\em Retarded Potentials and Time Domain Boundary Integral Equations:
  A Road Map}, volume~50.
\newblock Springer, 2016.

\bibitem{steinbach2007numerical}
O.~Steinbach.
\newblock {\em Numerical approximation methods for elliptic boundary value
  problems: finite and boundary elements}.
\newblock Springer Science \& Business Media, 2007.

\bibitem{trefethen2013approximation}
L.~N. Trefethen.
\newblock {\em Approximation theory and approximation practice}, volume 128.
\newblock Siam, 2013.

\bibitem{wanner1996solving}
G.~Wanner and E.~Hairer.
\newblock {\em Solving ordinary differential equations II}.
\newblock Springer Berlin Heidelberg, 1996.

\bibitem{wanner1993solving}
G.~Wanner, E.~Hairer, and S.~P. Norsett.
\newblock {\em Solving ordinary differential equations I}.
\newblock Springer Berlin Heidelberg, 1993.

\end{thebibliography}

\end{document}